\documentclass[letterpaper,12pt,headings=normal,parskip=half]{scrartcl}

\usepackage[american]{babel}            % American English hyphenation
\usepackage[utf8]{inputenc}             % UTF-8 input encoding
\usepackage[T1]{fontenc}                % T1 font encoding
\usepackage{lmodern}                    % Improved font rendering

\usepackage{geometry}                   % Page geometry
\usepackage{fullpage}                   % Sets smaller margins
\usepackage{hyperref}                   % Hyperlinks in PDF
\usepackage[x11names]{xcolor}           % Comprehensive color package
\definecolor{TUBred}{RGB}{196,13,32}    % TU Berlin red
\hypersetup{
  colorlinks=true,
  linkcolor=navy,
  citecolor=gray,
  urlcolor=black,
  filecolor=black
}

\usepackage{amsmath, amssymb, amsthm}   % Math environments and symbols
\usepackage{mathtools}                  % Enhancements for amsmath
\usepackage{dsfont}                     % Double-struck fonts (e.g., \mathds{1})

\usepackage{graphicx}                   % Include images
\usepackage{caption}                    % Customize captions
\usepackage{subcaption}                 % Subfigures
\usepackage{booktabs}                   % Nicer tables
\usepackage{tabulary}                   % Flexible table columns

\usepackage[capitalize,noabbrev]{cleveref}
\usepackage{authblk}                    % Author and affiliation handling
\usepackage{xspace}                     % Smart spacing after macros
\usepackage{enumitem}                   % Customizable lists
\usepackage{tikz}                       % Drawing package
\usetikzlibrary{calc}                   % TikZ coordinate calculations
\addtokomafont{sectioning}{\rmfamily}   % Section fonts in roman
\mathcode`@="8000
{\catcode`\@=\active\gdef@{\mkern1mu}}

\DeclareMathOperator{\cdim}{cdim}
\DeclareMathOperator{\mdim}{mdim}
\DeclareMathOperator{\blocks}{\text{$b$}}

\theoremstyle{definition}
\newtheorem{theorem}{Theorem}
\newtheorem{definition}[theorem]{Definition}
\newtheorem{lemma}[theorem]{Lemma}
\newtheorem{example}[theorem]{Example}
\newtheorem{corollary}[theorem]{Corollary}
\newtheorem{problem}{Open Problem}

\definecolor{navy}{RGB}{38, 70, 83}
\definecolor{cyan}{RGB}{65, 138, 127}
\definecolor{gold}{RGB}{233, 196, 106}
\definecolor{apri}{RGB}{244, 162, 97}
\definecolor{clay}{RGB}{231, 111, 81}

\tikzstyle{labeledvertex} = [draw=black, thick, minimum size=4.5mm, inner sep=0pt, text centered, align=center, circle]
\tikzstyle{vertex} = [draw=black, fill=black, inner sep=0.75mm, circle]
\tikzstyle{edge} = [black, thick]

\makeatletter
\newlength{\negph@wd}
\DeclareRobustCommand{\negphantom}[1]{%
  \ifmmode
    \mathpalette\negph@math{#1}%
  \else
    \negph@do{#1}%
  \fi
}
\newcommand{\negph@math}[2]{\negph@do{$\m@th#1#2$}}
\newcommand{\negph@do}[1]{%
  \settowidth{\negph@wd}{#1}%
  \hspace*{-\negph@wd}%
}
\makeatother

\title{\Large The connectivity dimension of a graph}
\author{\normalsize Kurt Klement Gottwald\textsuperscript{1}, Tobias Hofmann\textsuperscript{2,}*}
\date{}
\affil{\footnotesize\textsuperscript{1}TU Chemnitz, \textsuperscript{2}TU Berlin\\
*Corresponding Author, \texttt{tobias.hfm@icloud.com}}

\begin{document}
\maketitle

\begin{abstract}
\noindent
\textbf{Abstract.} This article investigates the connectivity dimension of a graph. 
We introduce this concept in analogy to the metric dimension of a graph, providing a graph parameter that measures the heterogeneity of the connectivity structure of a graph. 
We fully characterize extremal examples and present explicit constructions of infinitely many graphs realizing any prescribed non-extremal connectivity dimension.
We also establish a general lower bound in terms of the graph's block structure, linking the parameter to classical notions from graph theory.
Finally, we prove that the problem of computing the connectivity dimension is NP-complete. \\



\end{abstract}

\section{Introduction}\label{sec:intro}

Identification problems in graphs concern methods to locate vertices uniquely based on structural information, such as distances to selected vertices or neighborhood profiles. Problems of this kind arise in applications ranging from network discovery to robot navigation, as surveyed by Tillquist, Frongillo, and Lladser~\cite{tillquist2023getting}. A classical setting is that of the metric dimension, introduced independently by Slater~\cite{slater1975leaves} and Harary and Melter~\cite{harary1976metric}. There, one asks for the smallest set of vertices that uniquely identifies all others by the ordered set of distances to the selected vertices. The motivation for that setting is to localize an object that can query the distances to landmarks to be placed in a network. Now suppose the object cannot query how far it is apart from landmarks, but how strong it is connected. This gives rise to the connectivity dimension\,---\,the graph invariant that we introduce in this article.

Whereas we refer to the monograph of Diestel~\cite{diestel2017graph} for basic graph theoretical terminology, let us proceed with concepts and notions that are of particular interest in what follows.
In~this article, graphs are nonempty, finite, simple, and undirected.
For two vertices~$v$ and~$w$ of a graph~$G$, we refer by~$\kappa(v,w)$ to the number of independent, also known as internally vertex disjoint, paths in $G$.
Hereby, we set ${{\kappa(v,v) \coloneqq \infty}}$ for all vertices~$v\in V(G)$. 
Given an ordered subset ${{W=\{w_1,\ldots,w_k\}}}$ of vertices of $G$, the \emph{connectivity representation} of a vertex ${{v\in V(G)}}$ is the vector ${{r_G(v,w)\coloneqq[\kappa(v,w_1),\ldots,\kappa(v,w_k)]}}$.
We sometimes refer to the vertices of $W$ as \emph{landmarks} and may omit the graph $G$ in the index of a connectivity representation if the context allows.
We say a vertex $w$ \emph{distinguishes} a set of vertices $U$ if $\kappa(v_1,w)\neq\kappa(v_2,w)$ for all pairs of vertices $v_1$ and $v_2$ in~$U$.
We call the set $W$ \emph{resolving} for the graph~$G$ if every pair of vertices is distinguished by some vertex~$w\in W$. Equivalently, $W$ is resolving if ${{r(v_1,W)=r(v_2,W)}}$ implies ${{v_1=v_2}}$ for all pairs of vertices $v_1$ and $v_2$ in~$V(G)$.
Note that while the order of entries in a connectivity representation matters, the resolving property depends only on the set itself.
Accordingly, we omit specifying orderings in many illustrations and arguments.
A resolving set of minimum cardinality is called a \emph{basis} of $G$ and we define the \emph{connectivity dimension} $\cdim(G)$ of $G$ as the cardinality of a basis of $G$. 
The metric dimension~$\mdim(G)$ is defined analogously. 
The only difference is that it involves distances~$d(v,w)$ instead of local connectivities~$\kappa(v,w)$ as entries of representation vectors.
Remarkably, despite this conceptual correspondence, the ratio between connectivity and metric dimension can be made arbitrarily small, as shown in \cref{cor:metricratio1}. In the localization setting described earlier, this means that connectivity representations may yield way more compact positional encodings. That said, the opposite effect can occur as well, as can be seen from \cref{cor:metricratio2}.

\begin{figure}
    \centering
    \begin{tikzpicture}[scale=1.15]
        \node[labeledvertex] (1) at (2.1,2.1) {\footnotesize{$v_1$}};
        \node[labeledvertex] (2) at (1.5,1.1) {\footnotesize{$v_2$}};
        \node[labeledvertex] (3) at (0.2,1.7) {\footnotesize{$v_3$}};
        \node[labeledvertex] (4) at (-0.9,1) {\footnotesize{$v_4$}};
        \node[labeledvertex] (5) at (0.1,-0.5) {\footnotesize{$v_5$}};
        \node[labeledvertex] (6) at (1.6,-0.2) {\footnotesize{$v_6$}};
        \node[labeledvertex] (7) at (3.2,-0.5) {\footnotesize{$v_7$}};
        \node[labeledvertex] (8) at (2.8,0.9) {\footnotesize{$v_8$}};

        \draw[edge] (1) -- (2);
        \draw[edge] (2) -- (3);
        \draw[edge] (2) -- (5);
        \draw[edge] (2) -- (6);
        \draw[edge] (3) -- (4);
        \draw[edge] (3) -- (5);
        \draw[edge] (3) -- (6);
        \draw[edge] (4) -- (5);
        \draw[edge] (5) -- (6);
        \draw[edge] (6) -- (7);
        \draw[edge] (6) -- (8);
        \draw[edge] (7) -- (8);
    \end{tikzpicture}
    \caption{A graph $G$ with $\cdim(G)=2$}
    \label{fig:introexample}
\end{figure}
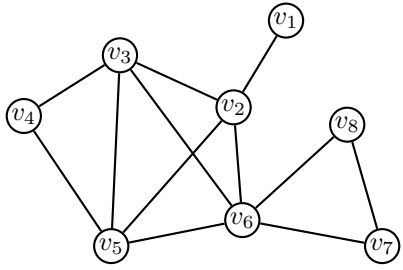

To get acquainted with the above notions, let us take a look at the graph~$G$ shown in \cref{fig:introexample}. 
The set ${{X\coloneqq\{v_1,v_4,v_8\}}}$ is not resolving, because ${{r(v_2,X) = [1,2,1] = r(v_3,X)}}$. 
One can verify that the set ${{Y\coloneqq\{v_2,v_5,v_7\}}}$ is resolving for the depicted graph. However, $Y$ is not of minimum cardinality, since ${{Z\coloneqq\{v_3,v_8\}}}$ is resolving as well. 
The corresponding connectivity representations are
\begin{align*}
    r(v_1,Z) &= [1,1],&&r(v_2,Z) = [3,1], \\
    r(v_3,Z) &= [\infty,1],&&r(v_4,Z) = [2,1], \\
    r(v_5,Z) &= [4,1],&&r(v_6,Z) = [3,2], \\
    r(v_7,Z) &= [1,2],&&r(v_8,Z) = [1,\infty].
\end{align*}
One can check that all vertex sets of cardinality one are not resolving for~$G$. 
More generally, the case ${{\cdim(G)=1}}$ is treated in \cref{thm:cdim_1}. 
Thus ${{\cdim(G)=2}}$. Note also that bases do not have to be unique. 
Further bases of $G$ are $\{v_3,v_7\}$, $\{v_5,v_7\}$, and $\{v_5,v_8\}$. 
Furthermore, for any graph~$G$ and any vertex set ${{W=\{w_1,\ldots,w_k\}\subseteq V(G)}}$, the entry in position~$k$ of the connectivity representation $r(w_k,W)$ is $\infty$ and $w_k$ is the only vertex in $V(G)$ having this entry at position $k$. 
In other words, the elements of $W$ itself are always distinguished and when verifying whether a given vertex set~$W$ is resolving for~$G$, one may focus just on~${{V(G)\setminus W}}$.

One interpretation of the connectivity dimension is as a measure of structural heterogeneity in graphs, as underlined by our results in Section~\ref{sec:prescribed}. 
Small connectivity dimension requires that connectivity varies heavily across the graph. In contrast, uniformly connected graphs have maximum connectivity dimension, as shown in \cref{thm:cdim_n-1}. Connectivity representations can also be seen as a way to enrich vertices with structural information or to encode vertices and their specific relations to others. 
Such representations are desirable, for instance, as input for machine learning models, as discussed by Hamilton, Ying, and Leskovec~\cite{hamilton2017representation}.

Whereas \cref{sec:prescribed} asks for graphs realizing a given connectivity dimension, in \cref{sec:blockstructure} we shift our perspective. 
Given a graph, what do we learn about its connectivity dimension from its classical graph theoretic properties?
As all local connectivities between vertices separated by an articulation of a connected graph are one, those vertices are a bottle neck for the information flow measured by the connectivity dimension.
This naturally motivates studying the relationship between the block structure of a graph and its connectivity dimension.
For the special but important case of bridges, we obtain the exact formula in \cref{cor:cdimofbridge}.
For general blocks, such precision cannot be expected.
However, in \cref{thm:cdimblocksratio}, we prove a general lower bound for the connectivity dimension solely depending on the number of blocks of a graph.

As Khuller, Raghavachari, and Rosenfeld~\cite{khuller1996landmarks} established for the metric dimension, we use a similar approach, in Section~\ref{sec:complexity}, to demonstrate that finding the connectivity dimension of an arbitrary graph is NP-complete. The authors of~\cite{khuller1996landmarks} also present a natural formulation of the metric dimension problem as a set cover problem, where each vertex corresponds to a subset covering all vertex pairs it distinguishes. This connection directly yields an $\mathcal{O}(\log n)$ approximation via the classical greedy algorithm, as shown by~Lov{\'a}sz~\cite{lovasz1975ratio}. Chartrand, Eroh, Johnson, and Oellermann~\cite{chartrand2000resolvability} further provide an integer programming formulation for computing the metric dimension, which can be easily adapted to the setting of the connectivity dimension. The approach by Sun~\cite{sun2023pp} to find metric representations by a position-aware graph neural network, achieving competitive computational results, could likewise be applied to determine connectivity representations.

\section{Graphs with prescribed connectivity dimension}\label{sec:prescribed}

First observe that $V(G)\setminus \{v\}$, where $v$ is an arbitrary vertex in $V(G)$, is resolving for any connected graph $G$. 
This yields the simple bounds
\begin{equation*}
    0\leq \cdim(G)\leq n-1,
\end{equation*}
where zero is only achieved if the graph consists of a single vertex. 
Before we show that the case $\cdim(G)=1$ is similarly restrictive,  we start by a technical lemma characterizing the behavior of the connectivity dimension upon disjoint union of graphs.

\begin{lemma}\label{lem:disjointunion}
    Let $G$ be a graph and let $G_1,\ldots,G_k$ be its connected components. 
    Let furthermore $G_1,\ldots,G_i$, $i\geq0$, be isolated vertices and $G_{i+1},\ldots,G_k$ be components of order at least two.
    Then there holds
    \begin{equation*}
        \cdim(G)=\max\{i-1,0\}+\sum_{j=i+1}^k\cdim(G_j).
    \end{equation*}
\end{lemma}

\begin{proof}
    One checks that the union of bases $B_j$ of $G_j$ together with the isolated vertices $G_1,\ldots,G_{i-1}$ yields a basis of $G$.
\end{proof}

\cref{lem:disjointunion} allows us to only consider connected graphs going forward.
The smallest possible connectivity dimension of graphs containing edges is one.
Let us show that there is only one connected graph that attains this.

\begin{theorem}\label{thm:cdim_1}
A connected graph $G$ has ${{\cdim(G)=1}}$ if and only if ${{G=K_2}}$.
\begin{proof}
Clearly, ${{\cdim(K_2)=1}}$. 
So we have to show that $K_2$ is the only connected graph with connectivity dimension one. 
Let us assume, for a contradiction, that there is a graph~$G$ on~$n\geq 3$ vertices having a resolving set of the form $W=\{w\}$ for some vertex~$w\in V(G)$. 
Furthermore, denote the vertices of~$G$ by~$w=v_1,v_2,\ldots,v_n$ such that $\kappa(v_2,w)\leq \ldots \leq \kappa(v_n,w)$. 
Since $\kappa(x,y)$ takes values only in $\{0,\ldots,n-1\}$, for any vertices~$x$ and~$y$, assuming that~$W$ is resolving for~$G$ means that
\begin{align*}
    r(v_1,W) &= [\kappa(v_1,w)] = [\infty],\\
    r(v_2,W) &= [\kappa(v_2,w)] = [1],\\
    &\hspace{2.2mm}\vdots\\
    r(v_n,W) &= [\kappa(v_n,w)] = [n-1].
\end{align*}
But $\kappa(v_n,w)=n-1$ says that~$v_n$ as well as~$w$ are adjacent to all other vertices. For~$n\geq 3$, we obtain two independent paths $v_2w$ and $v_2v_nw$, contradicting $\kappa(v_2,w) = 1$.
\end{proof}
\end{theorem}

The other extreme case $\cdim(G)=n-1$ can also be neatly characterized. 
As the following theorem shows, graphs attaining that bound are precisely the \emph{uniformly connected} graphs, studied by Beineke, Oellermann, and Pippert~\cite{beineke2002average} and Göring and Hofmann~\cite{hofmann2022uniformly}. 
They can be defined as graphs on at least $k+1$ vertices where each pair of vertices is connected by $k$ independent paths, and no pair is connected by more than $k$ independent paths. 

\begin{theorem}\label{thm:cdim_n-1}
A connected graph $G$ on $n$ vertices has ${{\cdim(G)=n-1}}$ if and only if $G$ is uniformly $k$-connected.
\begin{proof}
If $G$ is not uniformly $k$-connected, then there exist vertices ${{v_1,v_2,w\in V(G)}}$ with ${{\kappa(v_1,w)\neq \kappa(v_2,w)}}$. So~${{W\coloneqq V(G)\setminus\{v_1,v_2\}}}$ is resolving for $G$ and thus ${{\cdim(G)\leq n-2}}$.

Now suppose that $G$ is uniformly $k$-connected and consider an arbitrary vertex set ${{W\subseteq V(G)}}$ with ${{|W|\leq n-2}}$.
Denoting two vertices in ${{V(G)\setminus W}}$ by $v_1$ and $v_2$, we obtain
\begin{equation*}
    r(v_1,W)= [k,\ldots,k] = r(v_2,W).
\end{equation*}
So~$W$ cannot be resolving for~$G$ and ${{\cdim(G) = n-1}}$.
\end{proof}
\end{theorem}

This characterization shows that maximum connectivity dimension reflects high structural homogeneity. 
Another lower bound on the connectivity dimension can be given if a graph's maximum degree $\Delta$ is bounded. Since local connectivities cannot exceed $\Delta$, this yields a bound on the entries of connectivity representations and thus the number of vertices that can be distinguished by them.

\begin{theorem}\label{thm:cdimDelta}
The connectivity dimension of a graph $G$ on $n$ vertices with maximum degree $\Delta\geq 2$ satisfies
\begin{equation*}
    \log_\Delta \Big(\frac{n+1}{2}\Big)\leq\cdim(G).
\end{equation*}
\begin{proof}
If $\Delta$ is the maximum degree of $G$, then ${{\kappa(v,w)\leq\min\{\deg(v),\deg(w)\}\leq \Delta}}$. 
A connectivity basis $B$ of $G$ contains $\cdim(G)$ many elements. 
This allows for at most $\Delta^{\cdim(G)}$ connectivity representations $r(v,B)$ for vectors ${{v\notin B}}$. 
Since the vertices in $B$ can be distinguished from each other and from those not in $B$, we can distinguish at most $\cdim(G) + \Delta^{\cdim(G)}$ vertices. This quantity must be larger than $n$ for $B$ to be resolving. We obtain
\begin{equation}\label{eq:transcendental}
    n\leq \cdim(G) + \Delta^{\cdim(G)}.
\end{equation}
Using the inequality $x+1\leq \Delta^x$ for $\Delta\geq 2$ and $x\geq 1$, we derive $n\leq 2\Delta^{\cdim(G)}-1$ and thus
\begin{equation*}
    \log_\Delta \Big(\frac{n+1}{2}\Big)\leq \cdim(G). \qedhere
\end{equation*}
\end{proof}
\end{theorem}

Note that although this lower bound may be a rough estimate in general, it can be useful if~$n$ is large and~$\Delta$ is relatively small, that is, when $G$ is a large sparse graph. 
Further analytical improvements are possible.
For example, as discussed in the proof of \cref{thm:cdim_1}, no connectivity representation can contain the entries $1$ and $n-1$ at the same time.
However, solving \cref{eq:transcendental} without further simplifications requires numerical methods.

Having examined the range of possible dimension values, a key question remains:
does there exist a graph with connectivity dimension $k$ for every $0\leq k\leq n-1$?
And for $2\leq k\leq n-1$, can we even specify infinite families of graphs that attain that value?
Our answers to those questions involve the class of threshold graphs, first introduced by Chv{\'a}tal and Hammer~\cite{chvatal1977aggregations}.

\begin{definition}
    A \emph{threshold graph} on $n$ vertices is a graph that can be constructed from the empty graph by successively adding vertices according to two rules encoded by a binary sequence $x_1,\ldots, x_n$, $x_i\in\{0,1\}$.
    In step $i$, if $x_i=0$, an isolated vertex is added, and if $x_i=1$, a dominating vertex is added.
\end{definition}

Recall that two vertices $v_1$ and $v_2$ of a graph $G$ are called \emph{twins} if they have the same neighborhood in $G\setminus\{v_1,v_2\}$. 
Note that no third vertex can distinguish a pair of twins. 
Let $w\in V(G)\setminus\{v_1,v_2\}$ be any other vertex and $\mathcal{P}$ be any collection of $\kappa(v_1,w)$ independent $v_1$-$w$-paths in $G$. 
Then $\{P\setminus\{v_1\}\cup\{v_2\}:P\in\mathcal{P}\}$ is a collection of $\kappa(v_1,w)$ independent $v_2$-$w$-paths in $G$. 
It follows that $\kappa(v_2,w)\geq\kappa(v_1,w)$ and by symmetry, equality holds. 
Consequently, for every pair of twins in a graph, every connectivity basis has to contain at least one of them.

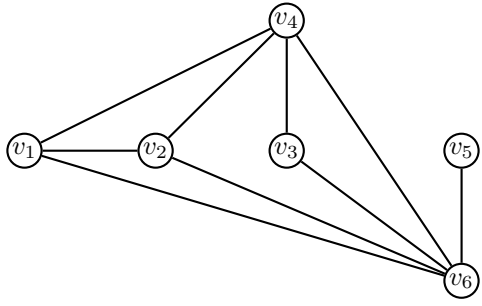
\begin{figure}
\centering
\begin{tikzpicture}[scale=1.15]

\node[labeledvertex] (1) at (0,0) {\footnotesize{$v_1$}};
\node[labeledvertex] (2) at (1.5,0) {\footnotesize{$v_2$}};
\node[labeledvertex] (3) at (3,0) {\footnotesize{$v_3$}};
\node[labeledvertex] (4) at (3,1.5) {\footnotesize{$v_4$}};
\node[labeledvertex] (5) at (5,0) {\footnotesize{$v_5$}};
\node[labeledvertex] (6) at (5,-1.5) {\footnotesize{$v_6$}};

\draw[edge] (2) to (1);
\draw[edge] (4) to (1);
\draw[edge] (4) to (2);
\draw[edge] (4) to (3);
\draw[edge] (6) to (1);
\draw[edge] (6) to (2);
\draw[edge] (6) to (3);
\draw[edge] (6) to (4);
\draw[edge] (6) to (5);

\end{tikzpicture}
\caption{A threshold graph generated by the sequence $0,1,0,1,0,1$.}
\label{fig:threshold}
\end{figure}
Take a look at \cref{fig:threshold} for an example of a threshold graph and its defining binary sequence.
In what follows, we denote such a sequence more compactly by $x_1^{k_1},x_2^{k_2},\ldots, x_m^{k_m}$, $x_i\in\{0,1\}$.
Herein, $x_i^{k_i}$ represents $k_i$ consecutive copies of $x_i$, for $i\in\{1,\ldots,m\}$, and we require $x_j\neq x_{j+1}$ for all $j\in\{1,\ldots,m-1\}$. Note that it is irrelevant whether the first entry of such a sequence is zero or one, as we add a singleton in both cases. So assuming $k_1\geq 2$ in sequences of length at least two is not a restriction.
Similarly, focusing on the case where $x_m=1$ is not truly a restriction, as this is necessary for a threshold graph to be connected.
We can infer the connectivity dimension of a disconnected graph from the connectivity dimensions of its connected components using \cref{lem:disjointunion}. 

\begin{theorem}\label{thm:cdimthreshold}
    Let $G$ be the threshold graph on $n$ vertices generated by the sequence $x_1^{k_1},x_2^{k_2},\ldots, x_m^{k_m}$, $x_i\in\{0,1\},x_m=1,k_1\geq 2$. Then
    \begin{equation*}
        \cdim(G)=\begin{cases}
            n-m   &\text{if }k_m > 1, \\
            n-m+1 &\text{if }k_m = 1.\vspace{-1ex}
        \end{cases}
    \end{equation*}
\end{theorem}
\begin{proof}
    Any two vertices corresponding to the same constant subsequence $x_i^{k_i}$ are twins. 
    Thus, any resolving set of $G$ has to contain at least all but one of the corresponding vertices.
    It follows that
    \begin{equation*}
        \cdim(G)\geq\sum_{i=1}^m(k_i-1)=\sum_{i=1}^mk_i-m=n-m.
    \end{equation*}

    For the converse inequality, let $W\subseteq V(G)$ for all $i\leq m$ contain all but one of the vertices corresponding to $x_i^{k_i}$, and let $k_m>1$. 
    Then $W$ is a set of size $n-m$. 
    We have to show that $W$ is a resolving set for $G$. 
    Consider any vertex $w$ corresponding to $x_m^{k_m}$ in $W$. 
    Let also $v\neq w$ be any vertex corresponding to $x_j^{k_j}$. 
    Because $w$ is a dominating vertex in $G$, we obtain\vspace{-1ex}
    \begin{align*}
        \kappa(v,w)&=\deg(v)=\!\sum_{i=j+1}^m\!\!x_i@k_i&&\hspace{-16mm}\text{if }x_j=0\quad\text{and}\\
        \kappa(v,w)&=\deg(v)=\sum_{i=0}^jk_i+\!\sum_{i=j+1}\!\!x_i@k_i&&\hspace{-16mm}\text{if }x_j=1.
    \end{align*}
    It follows that $\kappa(v_1,w)\neq\kappa(v_2,w)$ whenever $v_1$ and $v_2$ correspond to different constant subsequences.
    Thus, with $w\in W$ we can deduce which constant subsequence $x_i^{k_i}$ a vertex corresponds to, and since there is only one vertex corresponding to each such subsequence that is not an element of $W$, it follows that $W$ is resolving.
    
    It is now left to show that if $k_m=1$, we have $\cdim(G)=n-m+1$. Let first $m=2$, in which case $x_1=0$ and $x_2=1$. 
    Then $G$ is a star.
    In particular, it is uniformly \mbox{$1$-connected} and hence $\cdim(G)=n-1$. 
    Let now $m>2$ and denote the unique vertex corresponding to $x_m$ by $v_m$. Furthermore, let $v_{m-2}$ be any vertex corresponding to $x_{m-2}^{k_{m-2}}$ and let $w$ be any third vertex. 
    We already saw that $\kappa(v_m,w)=\deg(w)$. 
    The neighborhood of $w$ is contained in the neighborhood of $v_{m-2}$. 
    We thus also have $\kappa(v_{m-2},w)=\deg(w)$. 
    So $\kappa(v_m,w)=\kappa(v_{m-2},w)$ and thus no landmark other than $v_m$ itself can distinguish $v_m$ from the vertices corresponding to~$x_{m-2}^{k_{m-2}}$. 
    This means that any resolving set has to contain $v_m$ or every vertex corresponding to~$x_{m-2}^{k_{m-2}}$. 
    In any case, we get $\cdim(G)=n-m+1$.
\end{proof}

As argued in the beginning of the proof of \cref{thm:cdimthreshold}, just from knowing the count of twins of a graph, one gets the lower bound $\cdim(G)\geq n-m$. Herein, $m$ is the largest size of a set of vertices that are pairwise not twins.
In this sense, threshold graphs minimize the connectivity dimension.

\begin{example}\label{ex:house}
\begin{figure}
        \centering
    \begin{tikzpicture}[scale=0.96]

\node[labeledvertex] (1) at (0,0) {};
\node at (0.7,0) {$[2,2]$};
\node[labeledvertex, cyan, fill=cyan!30!white] (2) at (-1,-1) {\textcolor{cyan!70!black}{$w_1$}};
\node[labeledvertex] (3) at (1,-1) {};
\node at (1.7,-1) {$[4,3]$};
\node[labeledvertex,cyan, fill=cyan!30!white] (4) at (-1,-3) {\textcolor{cyan!70!black}{$w_2$}};
\node[labeledvertex] (5) at (1,-3) {};
\node at (1.7,-3) {$[3,3]$};

\draw[edge] (1) to (2);
\draw[edge] (1) to (3);
\draw[edge] (2) to (3);
\draw[edge] (2) to (4);
\draw[edge] (2) to (5);
\draw[edge] (3) to (4);
\draw[edge] (3) to (5);
\draw[edge] (4) to (5);

\end{tikzpicture}
        \caption{The house graph with a connectivity basis $\{w_1,w_2\}$ and corresponding connectivity representations}
        \label{fig:house}
    \end{figure}
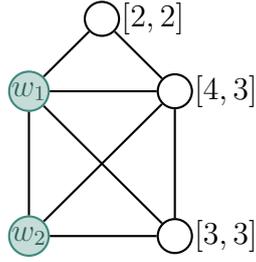
    Consider the threshold graph generated by the sequence $1,1,0,1,1$.
    We refer to it as the \emph{house graph}. 
    It contains two pairs of twins and so its connectivity dimension is at least two. 
    \cref{thm:cdimthreshold} provides the exact value, $n-m=5-3=2$.
    A possible connectivity basis is highlighted green in \cref{fig:house}.
\end{example}

\begin{example}
    It is not hard to see that threshold graphs are maximally locally connected, meaning that $\kappa(v,w)=\min\{\deg(v),\deg(w)\}$ for all pairs of vertices.
    The threshold graph encoded by sequence $x_1^{k_1},x_2^{k_2},\ldots, x_m^{k_m}$, $x_i\in\{0,1\},x_m=1,k_1\geq 2, k_m=1$, has therefore $m-1$ different local connectivities. 
    A naive extrapolation from this fact on the one hand, and \cref{thm:cdim_n-1} on the other extreme, might be that the connectivity dimension simply equals the order of a graph minus the number of distinct local connectivities.
    A graph of connectivity dimension~$n-2$ is depicted in \cref{fig:cdimn-2}. 
Beyond verifying its connectivity dimension by hand, one may recognize the graph as a union of two uniformly connected graphs joined by a bridge. 
Its dimension can be deduced from \cref{thm:cdim_n-1} anticipating the result of~\cref{cor:cdimofbridge}.
This example shows that the naive extrapolation is not valid. 
In the depicted case, $\cdim(G)=n-2$ despite that there are three distinct local connectivity values. 
So the connectivity dimension cannot be inferred solely by the count of local connectivity values that appear for a graph.
\begin{figure}
\centering
\begin{tikzpicture}[scale=1.35]

\node[labeledvertex] (1) at (-1,0) {\footnotesize{$v_1$}};
\node[labeledvertex] (2) at (0,1) {\footnotesize{$v_2$}};
\node[labeledvertex] (3) at (0,0) {\footnotesize{$v_3$}};
\node[labeledvertex] (4) at (0,-1) {\footnotesize{$v_4$}};
\node[labeledvertex] (5) at (1,0) {\footnotesize{$v_5$}};
\node[labeledvertex] (6) at (2.5,0) {\footnotesize{$v_6$}};
\node[labeledvertex] (7) at (3.5,1) {\footnotesize{$v_7$}};
\node[labeledvertex] (8) at (3.5,-1) {\footnotesize{$v_8$}};
\node[labeledvertex] (9) at (4.5,0) {\footnotesize{$v_9$}};

\draw[edge] (1) to (2);
\draw[edge] (1) to (3);
\draw[edge] (1) to (4);
\draw[edge] (2) to (3);
\draw[edge] (3) to (4);
\draw[edge] (2) to (5);
\draw[edge] (3) to (5);
\draw[edge] (4) to (5);
\draw[edge] (5) to (6);
\draw[edge] (6) to (7);
\draw[edge] (6) to (8);
\draw[edge] (7) to (9);
\draw[edge] (8) to (9);

\end{tikzpicture}
\caption{A graph of connectivity dimension $n-2$.}
\label{fig:cdimn-2}
\end{figure}
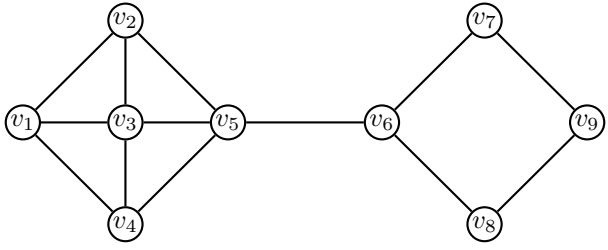
\end{example}

\cref{thm:cdimthreshold} now provides us with a method to construct infinite families of graphs having given connectivity dimension $k\geq 2$. Just take $n\in\mathbb{N}$ and choose $m$ such that $n-m+1=k$. For example, sequences of the form $1^2,0,1,\ldots,0,1$ yield graphs of connectivity dimension two. Sequences of the form $1^3,0,1,\ldots,0,1$ yield graphs of connectivity dimension three and so on.
Let us continue by investigating how the connectivity dimension behaves under taking subgraphs.
Clearly, if $G$ is a uniformly connected graph that is large enough and $H$ is a subgraph consisting just of a single vertex, then the ratio $\smash{\frac{\cdim(H)}{\cdim(G)}}$ can be made arbitrarily small.
The converse direction might not be that obvious.
\begin{corollary}\label{cor:subgraphratio}
    For every $\varepsilon>0$ there is a graph $G$ and an induced subgraph $H$ of $G$ such that
    \begin{equation*}
        \frac{\cdim(G)}{\cdim(H)}\leq\varepsilon.\vspace{-1ex}
    \end{equation*}
\end{corollary}

\begin{proof}
    Choose $G$ to be the threshold graph defined by the sequence $0,1,0,1,\ldots,0,1$, of order $2n, n\geq 1$. 
    Deleting all vertices corresponding to zeros in that sequence leaves us with $H=K_n$ as an induced subgraph of $G$. According to \cref{thm:cdimthreshold}, we have $\cdim(G)=2$ and $\cdim(H)=n-1.$
\end{proof}
We may also draw comparisons to the metric dimension.
\begin{corollary}\label{cor:metricratio1}
    For every $\varepsilon>0$ there is a graph $G$ such that
\begin{equation*}
    \frac{\cdim(G)}{\mdim(G)}\leq\varepsilon.\vspace{-1ex}
\end{equation*}
\begin{proof}
    Take again $G$ to be the threshold graph defined by the sequence $0,1,0,1,\ldots,0,1$, of order $n\geq 3$. We have $\cdim(G)=2$. For its metric dimension, note that $G$ has a dominating vertex and thus any distance in $G$ is at most two. Assume we have a metric basis of size $m\geq 1$. Using this basis, we can thus distinguish at most $m+2^m$ vertices. This is because, other than the $m$ vertices from the basis, every vertex has a metric representation vector consisting only of ones and twos. There are $2^m$ such vectors and we obtain that $m+2^m\geq n$. In particular, $m\to\infty$ as $n\to\infty$.
\end{proof}
\end{corollary}
The converse is also true.
\begin{corollary}\label{cor:metricratio2}
    For every $\varepsilon>0$ there is a graph $G$ such that
\begin{equation*}
    \frac{\mdim(G)}{\cdim(G)}\leq\varepsilon.\vspace{-1ex}
\end{equation*}
\begin{proof}
   Consider the path graph~$P_n$. We have $\mdim(P_n)=1$. This can be seen by using one of its vertices of degree one as single landmark. On the other hand, $P_n$ is uniformly connected, which means that $\cdim(P_n)=n-1$.
\end{proof}
\end{corollary}

\section{The connectivity dimension and the block structure of a graph}\label{sec:blockstructure}

Let $G$ be a connected graph and $Q$ be a connected subgraph of $G$ that is connected to $G-V(Q)$ via a cut vertex of $G$. 
Given vertices $v,w\in V(Q)$, we find that no path from $v$ to $w$ can leave $Q$ since it would have to reenter it eventually, and for this it would have to use the same cut vertex again.
Consequently, for all $W\subseteq V(G)$ it follows
\begin{equation}\label{eq:rofblock}
    r_Q(v,W\cap V(Q))=r_G(v,W\cap V(Q))=r_G(v,W)|_{V(Q)}.
\end{equation}
In words, $r_Q(v,W\cap V(Q))$ is the restriction of $r_G(v,W)$ to $Q$.

\begin{figure}
\centering
\begin{tikzpicture}[scale=1.1]

    \node at (2.8,2.1) {\textcolor{navy}{$B_1$}};
    \node at (-0.3,0.8) {\textcolor{navy}{$B_2$}};
    \node at (2.74,-0.1) {\textcolor{navy}{$B_3$}};
    
    \draw[thick, fill=none, draw=navy, rounded corners=2.5mm, densely dashed] (1.9,1.4) -- (0.16,2.18) -- (-1.43,1.1) -- (-0.1,-0.95) -- (2,-0.5) -- cycle;

    \draw[thick, fill=none, draw=navy, rounded corners=4mm, densely dashed] (0.95,0.93) -- (1.96,2.61) -- (2.62,2.3) -- (1.65,0.58) -- cycle;

    \begin{scope}[scale=1.8, shift={(-1.135,-0.031)}]
    \draw[thick, fill=none, draw=navy, rounded corners=6mm, densely dashed] (1.6,-0.2) -- (3.2,-0.5) -- (2.8,0.9) -- cycle;
    \end{scope}
    
    \node[labeledvertex] (1) at (2.1,2.1) {\footnotesize{$v_1$}};
    \node[labeledvertex] (2) at (1.5,1.1) {\footnotesize{$v_2$}};
    \node[labeledvertex] (3) at (0.2,1.7) {\footnotesize{$v_3$}};
    \node[labeledvertex] (4) at (-0.9,1) {\footnotesize{$v_4$}};
    \node[labeledvertex] (5) at (0.1,-0.5) {\footnotesize{$v_5$}};
    \node[labeledvertex] (6) at (1.6,-0.2) {\footnotesize{$v_6$}};
    \node[labeledvertex] (7) at (3.2,-0.5) {\footnotesize{$v_7$}};
    \node[labeledvertex] (8) at (2.8,0.9) {\footnotesize{$v_8$}};

    \draw[edge] (1) -- (2);
    \draw[edge] (2) -- (3);
    \draw[edge] (2) -- (5);
    \draw[edge] (2) -- (6);
    \draw[edge] (3) -- (4);
    \draw[edge] (3) -- (5);
    \draw[edge] (3) -- (6);
    \draw[edge] (4) -- (5);
    \draw[edge] (5) -- (6);
    \draw[edge] (6) -- (7);
    \draw[edge] (6) -- (8);
    \draw[edge] (7) -- (8);

    \begin{scope}[shift={(5.5,0)}]
    \node[labeledvertex, navy, minimum size=6mm] (9) at (2.1,2.1) {\textcolor{navy}{\footnotesize{$B_1$}}};
    \node[labeledvertex] (10) at (1.5,1.1) {\footnotesize{$v_2$}};
    \node[labeledvertex, navy, minimum size=6mm] (11) at (0.15,0.6) {\textcolor{navy}{\footnotesize{$B_2$}}};
    \node[labeledvertex] (12) at (1.6,-0.2) {\footnotesize{$v_6$}};
    \node[labeledvertex, navy, minimum size=6mm] (13) at (3,0.2) {\textcolor{navy}{\footnotesize{$B_3$}}};
    \end{scope}
    
    \draw[edge] (9) -- (10);
    \draw[edge] (10) -- (11);
    \draw[edge] (11) -- (12);
    \draw[edge] (12) -- (13);
\end{tikzpicture}
    \caption{The graph from \cref{fig:introexample} and its block graph.}
    \label{fig:blockgraphexample}
\end{figure}
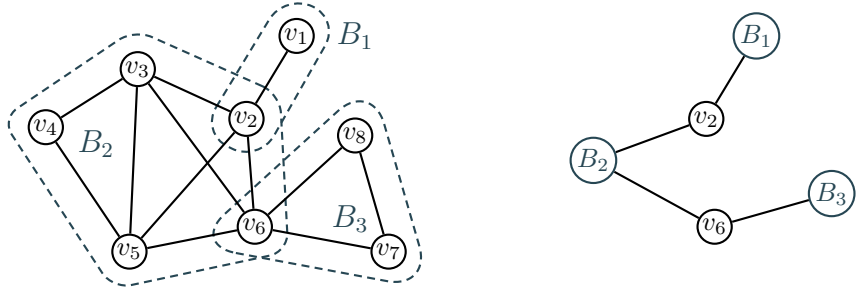

A \emph{block} of a graph~$G$ is a maximal connected subgraph of~$G$ that does not contain a cut vertex of that subgraph.  
Two different blocks can only intersect in at most one cut vertex of $G$. 
Note that \cref{eq:rofblock} is true whenever $Q$ is a block of $G$. 
Every block is either an isolated vertex, a bridge, or a maximal \mbox{$2$-connected} subgraph.
In view of \cref{lem:disjointunion}, we may disregard isolated vertices in what follows.
Denoting by $\mathcal{Q}$ the set of all blocks of $G$ and by $C$ the set of all cut vertices of $G$, the \emph{block graph} of $G$ is the bipartite graph on vertex set $\mathcal{Q}\cup C$ that has an edge between $v\in C$ and $Q\in\mathcal{Q}$ if and only if $v\in Q$.
An example is depicted in \cref{fig:blockgraphexample}.
We call all vertices of a block that are no cut vertices of $G$ \emph{inner} vertices of that block.
The number of blocks of a graph $G$ will be denoted by $\blocks(G)$. 
Since the blocks of connected graphs are always joined via cut vertices, the following lemma will be useful in what follows.

\begin{lemma}\label{lem:removeblock}
    Consider a graph $G$, two subgraphs $G_1$ and $G_2$ with $V(G_1)\cap V(G_2)=\{w\}$, and a connectivity basis $B$ of $G$. 
    If $w$ is a cut vertex of $G$, then $(B\cap V(G_1))\cup\{w\}$ is resolving for $G_1$.
\end{lemma}

\begin{proof}
    For any vertex $v_2\in V(G_2)\setminus \{w\}$, $\kappa_G(v_1,v_2)=1$ for all vertices $v_1\in V(G_1)\setminus \{w\}$, as $\{w\}$ separates $v_1$ and $v_2$ in $G$.
    In other words, no vertex in $V(G_2)\setminus \{w\}$ can distinguish any two vertices in $V(G_1)\setminus \{w\}$.
    Thus, for $B$ to be a connectivity basis of $G$, $B\cap V(G_1)$ has to distinguish any pair of vertices in $V(G_1)\setminus \{w\}$. So $(B\cap V(G_1))\cup\{w\}$ is indeed resolving for $G_1$.
\end{proof}

The first major goal of this section is to understand how the connectivity dimension behaves when joining two graphs via a bridge. 
This question is addressed by \cref{lem:cdimHjoin}, whose broader setting is illustrated in \cref{fig:Hjoin}. 
We rely on the following terminology to describe the details of our result. 
We say a graph $G$ \emph{is forcing a $\mathds{1}$ representation} if for every basis $B$ of $G$ there is a vertex $v$ of $G$ such that $r_G(v,B)=\mathds{1}=[1,\ldots,1]$.
Since $B$ is a basis, there can be at most one such vertex.
In general, $v$ depends on $B$.
Let for example $G$ be a tree. 
Then for any $v\in V(G)$ we have that $V(G)\setminus\{v\}$ is a basis and $r_G(v,V(G)\setminus\{v\})=\mathds{1}$.

\begin{lemma}\label{lem:cdimHjoin}
    Let $G_1,\ldots,G_k$ be connected graphs, on pairwise disjoint vertex sets, with at least two vertices each and let $\ell$ be the number graphs in $G_1,\ldots,G_k$ forcing a $\mathds{1}$ representation. 
    For $v_i\in G_i$, $i\in\{1,\ldots,k\}$, let $H$ be a connected graph on $V(H)=\{v_1,\ldots,v_k\}$. For the graph $G$ with $V(G)=\bigcup_{i=1}^k V(G_i)$ and $E(G)=E(H)@\cup\,\bigcup_{i=1}^kE(G_i)$, there holds
    \begin{equation*}
        \cdim(G)\leq\max\{\ell-1,0\}+\sum_{i=1}^k \cdim(G_i)
    \end{equation*}
    Herein, equality holds if $H$ is a tree.
\end{lemma}%
\begin{proof}
    \textbf{Case $\ell=0$:} Let $B_i$ be an arbitrary connectivity basis in $G_i$, $i\in\{1,\ldots,k\}$ such that for every vertex $u_i\in V(G_i)$ it holds that $r_{G_i}(u_i,B_i)\neq\mathds{1}$.
    Our goal is to show that then $B=\bigcup_{i=1}^k B_i$ is a resolving set in $G$.

\begin{figure}
    \centering
    \begin{tikzpicture}[scale=1.6]

    \draw[very thick, fill=white, draw=lightgray, rounded corners=2.5mm] (0.025,1.4) -- (-1.2,0.35) -- (-0.7,-1) -- (0.75,-1.25) -- (1.35,0.075) -- cycle;

    \draw[black, thick] (0,1) to ($(0,1)!6/10!(0.49,1.31)$);
    \draw[black, thick, dotted, rounded corners=2mm] ($(0,1)!6/10!(0.49,1.31)$) to (0.49,1.31);

    \draw[black, thick] (0,1) to ($(0,1)!6/10!(-0.5,1.3)$);
    \draw[black, thick, dotted, rounded corners=2mm] ($(0,1)!6/10!(-0.5,1.3)$) to (-0.5,1.3);
    
    \draw[black, thick] (0,1) to ($(0,1)!6/10!(-0.47,0.6)$);
    \draw[black, thick, dotted, rounded corners=2mm] ($(0,1)!6/10!(-0.47,0.6)$) to (-0.47,0.6);

    \draw[black, thick] (-0.5,-0.75) to ($(-0.5,-0.75)!6/10!(-1.1,-0.83)$);
    \draw[black, thick, dotted, rounded corners=2mm] ($(-0.5,-0.75)!6/10!(-1.1,-0.83)$) to (-1.1,-0.83);

    \draw[black, thick] (-0.5,-0.75) to ($(-0.5,-0.75)!6/10!(0.1,-0.75)$);
    \draw[black, thick, dotted, rounded corners=2mm] ($(-0.5,-0.75)!6/10!(0.1,-0.75)$) to (0.1,-0.75);

    \draw[black, thick] (-0.5,-0.75) to ($(-0.5,-0.75)!6/10!(-0.7,-0.23)$);
    \draw[black, thick, dotted, rounded corners=2mm] ($(-0.5,-0.75)!6/10!(-0.7,-0.23)$) to (-0.7,-0.23);

    \draw[black, thick] (1,0) to ($(1,0)!6/10!(1.4,0.4)$);
    \draw[black, thick, dotted, rounded corners=2mm] ($(1,0)!6/10!(1.4,0.4)$) to (1.4,0.4);

    \draw[black, thick] (1,0) to ($(1,0)!6/10!(1.34,-0.46)$);
    \draw[black, thick, dotted, rounded corners=2mm] ($(1,0)!6/10!(1.34,-0.46)$) to (1.34,-0.46);
    
    \node[labeledvertex, fill=white] (1) at (0,1) {\footnotesize{$v_1$}};
    
    \node[labeledvertex, fill=white] (2) at (-0.5,-0.75) {\footnotesize{$v_i$}};
    
    \node[labeledvertex, fill=white] (3) at (1,0) {\footnotesize{$v_k$}};

    \draw[edge] (1) -- (3);
    \draw[edge] (2) -- (3);

    \draw[thick, fill=none, draw=navy, rounded corners=2.5mm, densely dashed] (-0.9,1.3) -- (0,0.65) -- (1,1.35);

    \draw[thick, fill=none, draw=navy, rounded corners=2.5mm, densely dashed] (1.4,0.8) -- (0.65,0) -- (1.4,-0.9);

    \draw[thick, fill=none, draw=navy, rounded corners=2.5mm, densely dashed] (-1.3,-0.6) -- (-0.35,-0.475) -- (0,-1.3);
        
    \node at (0,0) {\textcolor{gray}{$H$}};
    \node at (0.05,1.5) {\textcolor{navy}{$G_1$}};
    \node at (-0.8,-1.1) {\textcolor{navy}{$G_i$}};
    \node at (1.525,0) {\textcolor{navy}{$G_k$}};
    
    \end{tikzpicture}
    \caption{Illustration of \cref{lem:cdimHjoin}}
    \label{fig:Hjoin}
\end{figure}

    Consider two distinct vertices $u_1\in V(G_{i_1})\setminus B_{i_1}$ and $u_2\in V(G_{i_2})\setminus B_{i_2}$, $i_1,i_2\in\{1,\ldots,k\}$.
    If $i_1=i_2$, then $r_{G_{i_1}}(u_1,B_{i_1})\neq r_{G_{i_1} }(u_2,B_{i_1})$, because $B_{i_1}$ is a resolving set in $G_{i_1}$.
    Then, by \cref{eq:rofblock}, $r_{G}(u_1,B)\neq r_{G}(u_2,B)$, as $v_{i_1}$ is a cut vertex of $G$.
    Let therefore~$i_1\neq i_2$.
    By construction, $r_G(u_1,B_{i_1})=r_{G_{i_1}}(u_1,B_{i_1})\neq\mathds{1}$ and $r_G(u_2,B_{i_2})=r_{G_{i_2}}(u_2,B_{i_2})\neq\mathds{1}$. 
    The only pair of vertices in $V(G_{i_1})\times V(G_{i_2})$ that potentially has local connectivity other than one is $(v_{i_1},v_{i_2})$ as any path from $V(G_{i_2})$ to $V(G_{i_1})$ has to pass through $v_{i_1}$ and $v_{i_2}$. 
    Therefore, if $u_2\neq v_{i_2}$, we have $r_G(u_2,B_{i_1})=\mathds{1}\neq r_G(u_1,B_{i_1})$. 
    If however $u_2=v_{i_2}$, then $v_{i_2}\not\in B_{i_2}$, and thus $r_G(u_1,B_{i_2})=\mathds{1}\neq r_G(u_2,B_{i_2})$. 
    So $B$ is resolving and $\cdim(G)\leq|B|=\sum_{i=1}^k\cdim(G_i)$.

    Let us verify that the inequality is attained if $H$ is a tree. 
    As such, $H$ is uniformly \mbox{$1$-connected}.
    Consider a basis $B$ in $G$. 
    We show that $B_i\coloneqq B\cap V(G_i)$ is a resolving set in $G_i$.
    For this, let $u_1,u_2\in V(G_i)$. 
    Since $B$ is resolving, we know that $r_G(u_1,B)\neq r_G(u_2,B)$. 
    If $v_i\notin\{u_1,u_2\}$, we immediately obtain that $r_G(u_1,B\setminus B_i)=r_G(u_2,B\setminus B_i)=\mathds{1}$. 
    However, since $H$ is a tree, the same holds true for $v_i\in\{u_1,u_2\}$. 
    By \cref{eq:rofblock}, $r_{G_i}(u_1,B_i)=r_G(u_1,B_i)\neq r_G(u_2,B_i)=r_{G_i}(u_2,B_i)$.
    In particular, $\cdim(G_i)\leq|B_i|$.
    Combining this yields\vspace{-1ex} 
    \begin{equation*}
        \cdim(G)=|B|=\sum_{i=1}^k|B_i|\geq\sum_{i=1}^k\cdim(G_i).\vspace{-1ex}
    \end{equation*}    
    \textbf{Case $\ell\geq1$:} Let $B_i$ be connectivity bases in $G_i$, $i\in\{1,\ldots,k\}$, respectively. 
    Furthermore, let the $B_i$ be chosen such that $r_{G_i}(w,B_i)=\mathds{1}$ for some vertex $w\in V(G_i)$ only if $G_i$ is forcing a $\mathds{1}$ representation.
    From each graph $G_i$ that is forcing a $\mathds{1}$ representation, we collect in $L$ the vertex $w_i\in V(G_i)$ with connectivity representation $r_{G_i}(w_i,B_i)=\mathds{1}$.
    Our goal is to show that for each $x\in L$, $B=\bigcup_{i=1}^kB_i\cup(L\setminus\{x\})$ is a resolving set of $G$.
    
    Analogously to the case $\ell=0$, we find a vertex in $B$ that distinguishes $u_1,u_2\in V(G)\setminus L$. 
    Furthermore, $r_G(u_1,B)\neq\mathds{1}=r_G(x,B)$ by \cref{eq:rofblock}. 
    Thus, $B$ is resolving and we obtain $\cdim(G)\leq|B|=\sum_{i=1}^k\cdim(G_i)+\ell-1$.

    To show that the bound is attained if $H$ is a tree, let a basis $B$ in $G$ be given.
    Analogously to the case $\ell=0$, $B\cap V(G_i)$ is a resolving set in $G_i$ for all $i\in\{1,\ldots,k\}$. 
    Furthermore, there can be at most one graph $G_{i_0}$ containing a vertex $u_0\in V(G_{i_0})$ with $r_{G_{i_0}}(u_0,B\cap V(G_{i_0}))=\mathds{1}$. 
    Take $i\neq i_0$ such that $G_i$ is forcing a $\mathds{1}$ representation. 
    There are at least $\ell-1$ of those. 
    We have that $r_{G_i}(u,B\cap V(G_i))\neq\mathds{1}$ for every $u\in V(G_i)$. 
    This means that $B\cap V(G_i)$ cannot be a basis. 
    Since it is resolving, we obtain that $\cdim(G_i)+1\leq|B\cap V(G_i)|$ and thus
    \begin{equation*}
        \cdim(G)=|B|=\sum_{i=1}^k B\cap V(G_i)\geq\sum_{i=1}^k\cdim(G_i)+\ell-1.\qedhere
    \end{equation*}
\end{proof}

\begin{example}
    Consider the graph on the left in \cref{fig:Hjoinbound}. None of the subgraphs $G_i$ is forcing a $\mathds{1}$ representation. So we are in the setting $\ell=0$ of \cref{lem:cdimHjoin} and obtain
    \begin{equation*}
        \cdim(G)\leq\cdim(G_1)+\cdim(G_2)+\cdim(G_3)=2+2+2=6.
    \end{equation*}
    Since $H$ is not a tree, the bound does not have to be tight. Indeed, a resolving set of size five is shown in \cref{fig:Hjoinbound}.
    
    Now consider the graph on the right.
    Again $\cdim(G_1)+\cdim(G_2)+\cdim(G_3)=6$, but this time the bound is attained. 
    The vertex of degree two is not contained in any connectivity basis of a house graph.
    In order to distinguish all vertices of the house graph, any resolving set of $G$ has to contain two other vertices of each house.
    A possible basis is illustrated in \cref{fig:Hjoinbound}.
    This also shows that for the above bound to be tight it is sufficient but not necessary for $H$ to be a tree. The following corollary now details the situation where $H$ is a bridge.

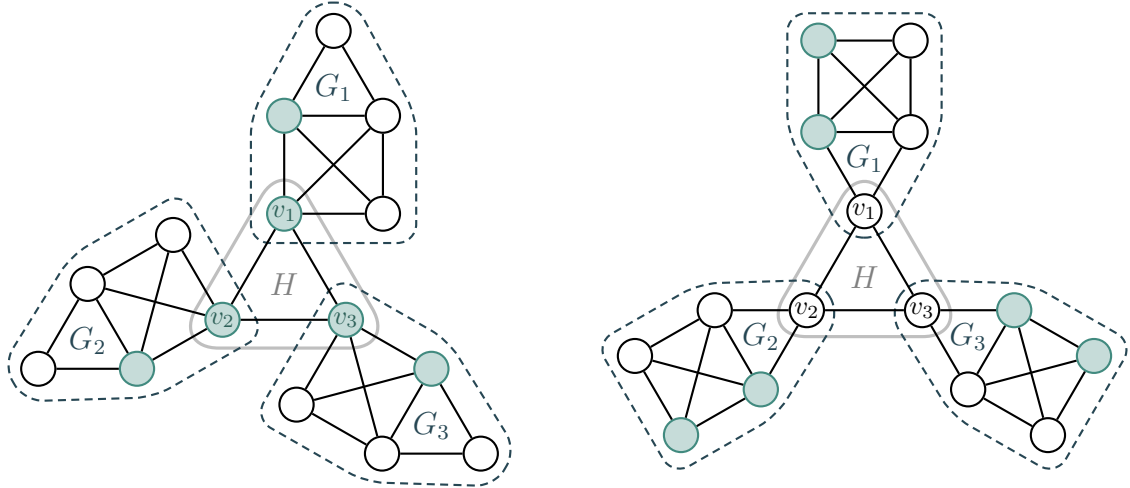
\begin{figure}
\centering
\begin{minipage}{.5\textwidth}
\centering
\begin{tikzpicture}[scale=0.375]

\begin{scope}[scale=2.25]
\draw[very thick, fill=white, draw=lightgray, rounded corners=6mm] (0,2) -- (-1.732,-1) -- (1.732,-1) -- cycle;
\end{scope}

\node at (0,0) {\textcolor{gray}{$H$}};

\begin{scope}[shift={(0,0.5)}]

\draw[thick, fill=none, draw=navy, rounded corners=2.5mm, densely dashed] (2.482,9.464) -- (0.982,9.464) -- (-1.168,5.464) -- (-1.168,0.964) -- (4.632,0.964) -- (4.632,5.464) -- cycle;

\node at (1.732,6.464) {\textcolor{navy}{$G_1$}};

\node[labeledvertex, cyan, fill=cyan!30!white] (G1bl) at (0,2) {\textcolor{cyan!70!black}{\footnotesize{$v_1$}}};
\node[labeledvertex] (G1br) at (3.464,2) {};
\node[labeledvertex, cyan, fill=cyan!30!white] (G1tl) at (0,5.464) {};
\node[labeledvertex] (G1tr) at (3.464,5.464) {};
\node[labeledvertex] (G1th) at (1.732,8.464) {};

\draw[edge] (G1bl) -- (G1br);
\draw[edge] (G1bl) -- (G1tl);
\draw[edge] (G1bl) -- (G1tr);
\draw[edge] (G1br) -- (G1tl);
\draw[edge] (G1br) -- (G1tr);
\draw[edge] (G1tl) -- (G1tr);
\draw[edge] (G1tl) -- (G1th);
\draw[edge] (G1tr) -- (G1th);
\end{scope}

\begin{scope}[rotate=120]
\begin{scope}[shift={(0,0.5)}]

\draw[thick, fill=none, draw=navy, rounded corners=2.5mm, densely dashed] (2.482,9.464) -- (0.982,9.464) -- (-1.168,5.464) -- (-1.168,0.964) -- (4.632,0.964) -- (4.632,5.464) -- cycle;

\node at (1.732,6.464) {\textcolor{navy}{$G_2$}};

\node[labeledvertex, cyan, fill=cyan!30!white] (G2bl) at (0,2) {\textcolor{cyan!70!black}{\footnotesize{$v_2$}}};
\node[labeledvertex] (G2br) at (3.464,2) {};
\node[labeledvertex, cyan, fill=cyan!30!white] (G2tl) at (0,5.464) {};
\node[labeledvertex] (G2tr) at (3.464,5.464) {};
\node[labeledvertex] (G2th) at (1.732,8.464) {};

\draw[edge] (G2bl) -- (G2br);
\draw[edge] (G2bl) -- (G2tl);
\draw[edge] (G2bl) -- (G2tr);
\draw[edge] (G2br) -- (G2tl);
\draw[edge] (G2br) -- (G2tr);
\draw[edge] (G2tl) -- (G2tr);
\draw[edge] (G2tl) -- (G2th);
\draw[edge] (G2tr) -- (G2th);
\end{scope}
\end{scope}

\begin{scope}[rotate=240]
\begin{scope}[shift={(0,0.5)}]

\draw[thick, fill=none, draw=navy, rounded corners=2.5mm, densely dashed] (2.482,9.464) -- (0.982,9.464) -- (-1.168,5.464) -- (-1.168,0.964) -- (4.632,0.964) -- (4.632,5.464) -- cycle;

\node at (1.732,6.464) {\textcolor{navy}{$G_3$}};

\node[labeledvertex, cyan, fill=cyan!30!white] (G3bl) at (0,2) {\textcolor{cyan!70!black}{\footnotesize{$v_3$}}};
\node[labeledvertex] (G3br) at (3.464,2) {};
\node[labeledvertex, cyan, fill=cyan!30!white] (G3tl) at (0,5.464) {};
\node[labeledvertex] (G3tr) at (3.464,5.464) {};
\node[labeledvertex] (G3th) at (1.732,8.464) {};

\draw[edge] (G3bl) -- (G3br);
\draw[edge] (G3bl) -- (G3tl);
\draw[edge] (G3bl) -- (G3tr);
\draw[edge] (G3br) -- (G3tl);
\draw[edge] (G3br) -- (G3tr);
\draw[edge] (G3tl) -- (G3tr);
\draw[edge] (G3tl) -- (G3th);
\draw[edge] (G3tr) -- (G3th);
\end{scope}
\end{scope}

\draw[edge] (G1bl) -- (G2bl);
\draw[edge] (G1bl) -- (G3bl);
\draw[edge] (G2bl) -- (G3bl);

\end{tikzpicture}
\end{minipage}%
\begin{minipage}{0.5\textwidth}
\centering
\begin{tikzpicture}[scale=0.35]

\begin{scope}[scale=2.25]
\draw[very thick, fill=white, draw=lightgray, rounded corners=6mm] (0,2) -- (-1.732,-1) -- (1.732,-1) -- cycle;
\end{scope}

\node at (0,0) {\textcolor{gray}{$H$}};

\begin{scope}[shift={(0,0.5)}]

\draw[thick, fill=none, draw=navy, rounded corners=2.5mm, densely dashed] (-0.75,1) -- (0.75,1) -- (2.9,5) --(2.9,9.5) -- (-2.9,9.5) -- (-2.9,5) -- cycle;

\node at (0,4) {\textcolor{navy}{$G_1$}};

\node[labeledvertex] (v1) at (0,2) {\footnotesize{$v_1$}};
\node[labeledvertex, cyan, fill=cyan!30!white] (G1bl) at (-1.732,5) {};
\node[labeledvertex] (G1br) at (1.732,5) {};
\node[labeledvertex, cyan, fill=cyan!30!white] (G1tl) at (-1.732,8.464) {};
\node[labeledvertex] (G1tr) at (1.732,8.464) {};
\draw[edge] (v1) -- (G1bl);
\draw[edge] (v1) -- (G1br);
\draw[edge] (G1bl) -- (G1br);
\draw[edge] (G1bl) -- (G1tl);
\draw[edge] (G1bl) -- (G1tr);
\draw[edge] (G1br) -- (G1tl);
\draw[edge] (G1br) -- (G1tr);
\draw[edge] (G1tl) -- (G1tr);
\end{scope}

\begin{scope}[rotate=120]
\begin{scope}[shift={(0,0.5)}]

\draw[thick, fill=none, draw=navy, rounded corners=2.5mm, densely dashed] (-0.75,1) -- (0.75,1) -- (2.9,5) --(2.9,9.5) -- (-2.9,9.5) -- (-2.9,5) -- cycle;

\node at (0,4) {\textcolor{navy}{$G_2$}};

\node[labeledvertex] (v2) at (0,2) {\footnotesize{$v_2$}};
\node[labeledvertex, cyan, fill=cyan!30!white] (G2bl) at (-1.732,5) {};
\node[labeledvertex] (G2br) at (1.732,5) {};
\node[labeledvertex, cyan, fill=cyan!30!white] (G2tl) at (-1.732,8.464) {};
\node[labeledvertex] (G2tr) at (1.732,8.464) {};
\draw[edge] (v2) -- (G2bl);
\draw[edge] (v2) -- (G2br);
\draw[edge] (G2bl) -- (G2br);
\draw[edge] (G2bl) -- (G2tl);
\draw[edge] (G2bl) -- (G2tr);
\draw[edge] (G2br) -- (G2tl);
\draw[edge] (G2br) -- (G2tr);
\draw[edge] (G2tl) -- (G2tr);
\end{scope}
\end{scope}

\begin{scope}[rotate=240]
\begin{scope}[shift={(0,0.5)}]

\draw[thick, fill=none, draw=navy, rounded corners=2.5mm, densely dashed] (-0.75,1) -- (0.75,1) -- (2.9,5) --(2.9,9.5) -- (-2.9,9.5) -- (-2.9,5) -- cycle;

\node at (0,4) {\textcolor{navy}{$G_3$}};

\node[labeledvertex] (v3) at (0,2) {\footnotesize{$v_3$}};
\node[labeledvertex, cyan, fill=cyan!30!white] (G3bl) at (-1.732,5) {};
\node[labeledvertex] (G3br) at (1.732,5) {};
\node[labeledvertex, cyan, fill=cyan!30!white] (G3tl) at (-1.732,8.464) {};
\node[labeledvertex] (G3tr) at (1.732,8.464) {};
\draw[edge] (v3) -- (G3bl);
\draw[edge] (v3) -- (G3br);
\draw[edge] (G3bl) -- (G3br);
\draw[edge] (G3bl) -- (G3tl);
\draw[edge] (G3bl) -- (G3tr);
\draw[edge] (G3br) -- (G3tl);
\draw[edge] (G3br) -- (G3tr);
\draw[edge] (G3tl) -- (G3tr);
\end{scope}
\end{scope}

\draw[edge] (v1) -- (v2);
\draw[edge] (v1) -- (v3);
\draw[edge] (v2) -- (v3);

\end{tikzpicture}
\end{minipage}%
\caption{Two examples illustrating \cref{lem:cdimHjoin}, one where the upper bound is not attained on the left and one where it is actually an equality on the right.}
\label{fig:Hjoinbound}
\end{figure}
\end{example}

\begin{corollary}\label{cor:cdimofbridge}
    Let $G_1$ and $G_2$ be two connected graphs and let $v_1\in G_1$ and $v_2\in G_2$ be arbitrary vertices. Furthermore, let $G=(V(G_1)\cup V(G_2),E(G_1)\cup E(G_2)\cup v_1v_2)$ be the graph that arises by joining $G_1$ and $G_2$ via the edge $v_1v_2$, which is then a bridge in $G$. Then
    \begin{equation*}
        \cdim(G)=\begin{cases}
        \cdim(G_1)+\cdim(G_2)+1 &\text{ if $G_1$ and $G_2$ force a $\mathds{1}$ representation},\\
            \cdim(G_1)+\cdim(G_2) &\text{ otherwise}.
        \end{cases}
    \end{equation*}
\end{corollary}

\begin{proof}
    The only case that is not covered by \cref{lem:cdimHjoin} is that where $G_1$ or $G_2$ are single vertices. But it is easy to verify this as well.
\end{proof}

A simple, yet important special case of this result arises if $G_1$ is an arbitrary graph and~$G_2$ a single vertex.

\begin{corollary}\label{cor:attachaleaf}
    Attaching a leaf to a graph $G$ affects the connectivity dimension if and only if $G$ is forcing a $\mathds{1}$ representation.
    If it does, $\cdim(G)$ increases by one.
    In particular, it does not matter where the leaf is attached.
\end{corollary}

We conclude this section with a relation between the number of blocks of a graph and its connectivity dimension.

\begin{theorem}\label{thm:cdimblocksratio}
    For a connected graph $G$ on at least two vertices there holds 
    \begin{equation*}
        \cdim(G)\geq\frac{\blocks(G)+1}{2}.   
    \end{equation*}
\end{theorem}

\begin{proof}
    Note that, following \cref{cor:attachaleaf}, if we can find a counterexample $H$ to our claim, then either $H$ is forcing a $\mathds{1}$ representation, or we can attach a leaf to any vertex of $H$ to obtain a graph $H'$ with $\blocks(H')=\blocks(H)+1$ and
    \begin{equation*}
        \cdim(H')=\cdim(H)<\frac{\blocks(H)+1}{2}<\frac{\blocks(H')+1}{2}.
    \end{equation*}
    Thus, $H'$ is a counterexample as well.
    So, assuming the assertion is false, then there exists a counterexample that forces a $\mathds{1}$ representation.
    
    Let $H$ be a counterexample that forces a $\mathds{1}$ representation and that has minimum number of blocks possible.
    First, consider the case where $H$ consists of two subgraphs $H_1$ and $H_2$, on at least two vertices each, that are joined via a bridge.
    Then $\blocks(H)=\blocks(H_1)+\blocks(H_2)+1$ and $\blocks(H_1),\blocks(H_2)\geq1$.
    If $H_i$, $i\in\{1,2\}$, was a counterexample to the assertion of the theorem, then, similar as above, we could construct another counterexample $H_i'$ by attaching a leaf to $H_i$ that has fewer blocks than $H$.
    So, as neither $H_1$ nor $H_2$ are counterexamples, they satisfy $\cdim(H_i)\geq\frac{\blocks(H_i)+1}{2}$.
    By \cref{cor:cdimofbridge}, we obtain
    \begin{align*}
    \frac{\blocks(H_1)+\blocks(H_2)+2}{2} &= \frac{\blocks(H)+1}{2} \\[0.5ex]
    &>\cdim(H) \\[1.5ex]
    &\geq\cdim(H_1)+\cdim(H_2)\\[0.5ex]
    &\geq\frac{\blocks(H_1)+1}{2}+\frac{\blocks(H_2)+1}{2}=\frac{\blocks(H_1)+\blocks(H_2)+2}{2},
    \end{align*}
    which is a contradiction.
    It follows that $H_1$ or $H_2$ contains at most one vertex, that is, every bridge in $H$ is incident to a leaf.
    Hence, every block that does not contain a leaf is \mbox{$2$-connected}.
 
    Let $T$ be the block graph of $H$ and let $B$ be a connectivity basis in $H$.
    Also, denote by $v\in V(H)\setminus B$ the unique vertex with $r_H(v,B)=\mathds{1}$.
    Furthermore, let $R,L\in V(T)$ be two blocks of $H$ where $v\in R$ and $L$ is of maximal distance $2d$ from $R$ in $T$.
    
    Consider first the case where $d=1$.
    Then $R$ has nonempty intersection with every block in $H$.
    Since $K_2$ is no counterexample, \cref{thm:cdim_1} says that $\cdim(H)\geq2$.
    Since $H$ is chosen as a counterexample, $\blocks(H)\geq 4$.
    
    Suppose that $R=K_2$.
    Then $V(R)=\{v,w\}$ contains a leaf of $H$, and this leaf has to be $v$.
    Every block other than $R$ is connected to $R$ via $w$. 
    In fact, $T$ is a star with the cut vertex $w$ as its center.
    Then by \cref{thm:cdim_1}, the interior of every block other than $R$ has to have a landmark, because $w$ alone is not enough to distinguish the entire block. 
    Thus, using $\blocks(H)\geq4$, we obtain $\cdim(H)\geq\blocks(H)-1>\frac{\blocks(H)+1}{2}$, and so $H$ is not a counterexample.
    
    Now let $R$ be \mbox{$2$-connected}.
    Then $R$ does not contain a landmark, because if there was one, $v\in R$ would not have representation $r_H(v,B)=\mathds{1}$.
    However, since $2d=2$ is the maximum distance of any block to $R$ in $T$, we know that all cut vertices of $H$ are vertices in $R$. 
    Hence, every landmark in $H$ is an inner vertex of its respective block, and every block other than $R$ has to contain such a landmark.
    In particular, $\cdim(H)\geq\blocks(H)-1>\frac{\blocks(H)+1}{2}$.

    So let us focus on the case where $d>1$.
    Note that $T$ is a tree and therefore $L$ can be chosen to be a leaf in $T$. 
    We denote by $(R=Q_0,v_0,Q_1,v_1,\ldots,Q_{d-1},v_{d-1},Q_{d}=L)$ the shortest path from $R$ to $L$ in $T$.
    Let $L=L_1,L_2,\ldots,L_k$ be all leaves of $T$ that are adjacent to $Q_{d-1}$, excluding $R$ as the case may be.
    If $L_i$ represents a $K_2$ in $H$, for any $i$, then its inner vertex has to be a landmark, since only $v$ can have the connectivity representation $\mathds{1}$.
    Otherwise, $L_i$ is \mbox{$2$-connected}. 
    But then it also has to contain an inner vertex that is a landmark, as, by \cref{thm:cdim_1}, its cut vertex alone does not distinguish the entire block.
    The graph $H_s$ that we obtain from $H$ by deleting $Q_{d-1}\setminus\{v_{d-2}\},L_1,\ldots,L_k$ has $k+1$ fewer blocks than $H$, and by \cref{lem:removeblock}, we find that
    \begin{equation*}
        \cdim(H_s)\leq \cdim(H)-k+1<\frac{b_H+1}{2}-k+1=\frac{b_{H_s}+1}{2}+\frac{3}{2}-\frac{k}{2}.
    \end{equation*}
    Thus, if $k\geq3$, then $H_s$ is a counterexample, and by attaching a leaf if necessary, we obtain a counterexample $H_s'$ that forces a $\mathds{1}$ representation which is smaller than $H$.
    
    If $k=1$, then $Q_{d-1}$ has an inner vertex, since it is \mbox{$2$-connected} and thus of order at least three.
    Therefore, it has to contain a landmark as well, because otherwise this inner vertex would have connectivity vector $\mathds{1}$.
    In this case, \cref{lem:removeblock} we obtain
    \begin{equation*}
        \cdim(H_s)\leq\cdim(H)-2+1<\frac{b_H+1}{2}-1=\frac{b_{H_s}+1}{2},
    \end{equation*}
    as we delete two landmarks, one in $L\setminus Q_{d-1}$ and one in $Q_{d-1}$, which clearly cannot be the same. Consequently, $H_s$, plus a leaf if necessary, would once again be a smaller counterexample $H_s'$ that forces a $\mathds{1}$ representation.

    The case where $k=2$ remains.
    Then either $Q_{d-1}$ has an inner vertex and therefore a landmark, or it is a triangle, each of whose vertices is an articulation of $H$.
    In the latter~case, the two leaves $L_1$ and $L_2$ both contain an inner vertex and thus a landmark each.
    Yet, since $L_1$ and $L_2$ are disjoint, no landmark in one can distinguish any pair of vertices in the other.
    Thus, $B$ has to contain at least one further landmark in $Q_{d-1}\cup L_1\cup L_2$.
    So we find
    \begin{equation*}
        \cdim(H_s)\leq\cdim(H)-3+1<\frac{b_H+1}{2}-2=\frac{b_{H_s}+1}{2}-\frac{1}{2}<\frac{b_{H_s}+1}{2}
    \end{equation*}
    and, as before, either $H_s$ or $H_s'$ is a counterexample that forces a $\mathds{1}$ representation, which is smaller than $H$.
\end{proof}

The complete graph $K_n$ provides an example where the ratio $\smash{\frac{\cdim(G)}{\blocks(G)}}$ becomes arbitrarily large.
In this case, $\cdim(K_n)=n-1$, and $\blocks(K_n)=1$.
It is not that obvious what the smallest possible ratio is.
\cref{thm:cdimblocksratio} says that it is at least $\frac{1}{2}$. The following construction, which achieves a ratio of $\frac{2}{3}$, for an infinite family of graphs, is yet the best known to us.

\begin{example}\label{ex:triangles}
    Consider the graph $T_b$ obtained by joining $(b-1)$ triangles in a path-like manner and attaching a leaf at one end, which is illustrated in \cref{fig:Tn} for $b=6$.
    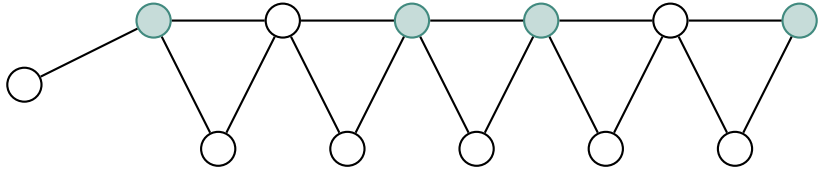
\begin{figure}
    \centering
    \begin{tikzpicture}[scale=1.7]
    \node[labeledvertex] (1) at (1,0.5) {};
    \node[labeledvertex, cyan, fill=cyan!30!white] (2) at (2,1) {};
    \node[labeledvertex] (3) at (3,1) {};
    \node[labeledvertex, cyan, fill=cyan!30!white] (4) at (4,1) {};
    \node[labeledvertex, cyan, fill=cyan!30!white] (5) at (5,1) {};
    \node[labeledvertex] (6) at (6,1) {};
    \node[labeledvertex, cyan, fill=cyan!30!white] (7) at (7,1) {};
    \node[labeledvertex] (8) at (2.5,0) {};
    \node[labeledvertex] (9) at (3.5,0) {};
    \node[labeledvertex] (10) at (4.5,0) {};
    \node[labeledvertex] (11) at (5.5,0) {};
    \node[labeledvertex] (12) at (6.5,0) {};

    \draw[edge] (1) -- (2);
    \draw[edge] (2) -- (3);
    \draw[edge] (3) -- (4);
    \draw[edge] (4) -- (5);
    \draw[edge] (5) -- (6);
    \draw[edge] (6) -- (7);
    \draw[edge] (2) -- (8);
    \draw[edge] (3) -- (8);
    \draw[edge] (3) -- (9);
    \draw[edge] (4) -- (9);
    \draw[edge] (4) -- (10);
    \draw[edge] (5) -- (10);
    \draw[edge] (5) -- (11);
    \draw[edge] (6) -- (11);
    \draw[edge] (6) -- (12);
    \draw[edge] (7) -- (12);
    
    \end{tikzpicture}
    \caption{The graph $T_6$ has six blocks and connectivity dimension four. The marked vertices form a connectivity basis.}
    \label{fig:Tn}
\end{figure}
The graph $T_b$ has $b$ blocks and has connectivity dimension
\begin{equation*}
    \cdim(T_b)=\begin{cases}
        \frac{2}{3}b&\text{if } 3\mid b,\\
        \frac{2}{3}b+\frac{1}{3}&\text{if } 3\mid b-1, \\
        \frac{2}{3}b+\frac{2}{3}&\text{if } 3\mid b-2.
    \end{cases}
\end{equation*}
\end{example}

\section{The complexity of determining the connectivity dimension}\label{sec:complexity}

Following the strategy of the proof in the appendix of Khuller, Raghavachari, and Rosenfeld~\cite{khuller1996landmarks}, which concerns the metric dimension, we show that 3-SAT can be reduced to determining the connectivity dimension.
For an arbitrary 3-SAT input $S$ on variables $X_1,\ldots,X_n$ and clauses $C_1,\ldots, C_m$, we construct a graph $G(S)$ whose connectivity dimension tells us whether $S$ is satisfiable. Without loss of generality, we may assume that for each variable $X_i$ there is a clause in which it appears. For each variable $X_i$ we construct a gadget $G(X_i)=(V(X_i),E(X_i))$ and for each clause~$C_j$ we construct a gadget $G(C_j)=(V(C_j),E(C_j))$ as depicted in \cref{fig:gadgets}.

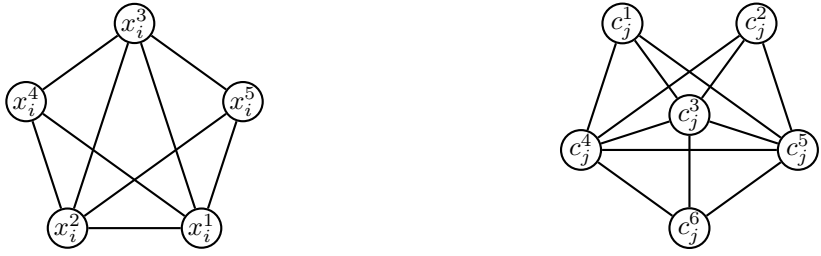
\begin{figure}[ht]
    \centering
        \begin{minipage}{.45\textwidth}
        \centering
    \begin{tikzpicture}[scale=1, rotate=18]

    \node[labeledvertex] (4) at (144:1.5) {\footnotesize{$x_i^4$}};
    \node[labeledvertex] (3) at (72:1.5) {\footnotesize{$x_i^3$}};
    \node[labeledvertex] (5) at (0:1.5) {\footnotesize{$x_i^5$}};
    \node[labeledvertex] (1) at (288:1.5) {\footnotesize{$x_i^1$}};
    \node[labeledvertex] (2) at (216:1.5) {\footnotesize{$x_i^2$}};

    \draw[edge] (1) -- (2);
    \draw[edge] (1) -- (3);
    \draw[edge] (1) -- (4); 
    \draw[edge] (1) -- (5);
    \draw[edge] (2) -- (3);
    \draw[edge] (2) -- (4);
    \draw[edge] (2) -- (5);
    \draw[edge] (3) -- (4);
    \draw[edge] (3) -- (5);

    \end{tikzpicture}
    \end{minipage}
    \begin{minipage}
    {0.45\textwidth}
        \centering
            \begin{tikzpicture}[scale=1, rotate=198]

    \node[labeledvertex] (4) at (144:1.5) {\footnotesize{$c_j^5$}};
    \node[labeledvertex] (3) at (0,0) {\footnotesize{$c_j^3$}};
    \node[labeledvertex] (5) at (0:1.5) {\footnotesize{$c_j^4$}};
    \node[labeledvertex] (1) at (288:1.5) {\footnotesize{$c_j^1$}};
    \node[labeledvertex] (2) at (216:1.5) {\footnotesize{$c_j^2$}};
    \node[labeledvertex] (6) at (72:1.5) {\footnotesize{$c_j^6$}};
    
    \draw[edge] (1) -- (3);
    \draw[edge] (1) -- (4);
    \draw[edge] (1) -- (5); 
    \draw[edge] (2) -- (3);
    \draw[edge] (2) -- (4);
    \draw[edge] (2) -- (5);
    \draw[edge] (3) -- (4);
    \draw[edge] (3) -- (5);
    \draw[edge] (4) -- (5);
    \draw[edge] (3) -- (6);
    \draw[edge] (4) -- (6);
    \draw[edge] (5) -- (6);

    \end{tikzpicture}
    \end{minipage}
    \caption{Gadget $G(X_i)$ for variable $X_i$ on the left and $G(C_j)$ for clause $C_j$ on the right.}
    \label{fig:gadgets}
\end{figure}

The construction of $G(S)$, which is illustrated in \cref{fig:ClauseAndVariables}, begins by adding all gadgets of occurring variables and clauses in~$S$ as isolated components. 
We then connect them by further edges as follows.
If the variable~$X_i$ occurs as a positive literal in the clause~$C_j$, we add the edges $\{c_j^1,x_i^1\}$, $\{c_j^2,x_i^1\}$ and $\{c_j^2,x_i^2\}$. 
If~$X_i$ occurs as a negative literal in $C_j$, we add the edges $\{c_j^1,x_i^1\}$, $\{c_j^1,x_i^2\}$ and $\{c_j^2,x_i^2\}$. 
For all $k$ where $X_k$ does not occur in $C_j$, neither as positive nor as negative literal, we do not add any edges between the respective gadgets. 

Lastly, we add the edges $\{c_j^1,c_k^1\},\{c_j^1,c_k^2\}, \{c_j^2,c_k^1\}$ and $\{c_j^2,c_k^2\}$ for all pairs of clauses $C_j\neq C_k$.
Note that if $G(S)$ is not connected then $S$ can be split into $S_1\wedge S_2$ such that no variable appears in both $S_1$ and $S_2$.
Then $S$ is satisfiable if and only if $S_1$ and $S_2$ are satisfiable.
Due to \cref{lem:disjointunion}, we can focus on instances $S$ where this is not possible and for which $G(S)$ is thus connected.

For the graph $G(S)$ we find the local connectivities 
\begin{align}
    \kappa(x_i^a,x_i^b)&\begin{cases}
        > 4 &\text{if } (a,b)=(1,2),\\
        =4 &\text{if } (a,b)=(1,3)\text{ or }(a,b)=(2,3),\\
        =3 &\text{otherwise,}
    \end{cases}\label{eq:samevargadget}\\[1ex]
    \kappa(c_j^a,c_j^b)&\begin{cases}
        >5 &\text{if } (a,b)=(1,2),\\
        =5 &\text{if } \{a,b\}\subseteq\{3,4,5\},\\
        =4 &\text{if } a\in\{1,2\},~b\in\{3,4,5\},\\
        =3 &\text{otherwise},
    \end{cases}\label{eq:sameclagadget}
\end{align}

assuming $a<b$.
Note that it is enough to consider $a<b$ since local connectivity is symmetric and $\kappa(x_i^a,x_i^a)=\kappa(c_j^a,c_j^a)=\infty$.

The removal of $x_i^1$ and $x_i^2$ disconnects $G(S)$ from $G(X_i)$. 
Similarly, the removal of $c_j^1$ and $c_j^2$ disconnects $G(S)$ from $G(C_j)$.
We conclude that

\begin{align}\label{eq:diffgadget}
\begin{split}
     \kappa(x_i^a,w)=2 &\text{ if } a\geq3 \text{ and } w\notin V(X_i),\\
     \kappa(x_i^a,w)\geq2 &\text{ if } a\in\{1,2\} \text{ and } w\notin V(X_i),\\
     \kappa(c_j^b,w)=2 &\text{ if } b\geq3 \text{ and } w\notin V(C_j)\\
     \kappa(c_j^b,w)\geq2 &\text{ if } b\in\{1,2\} \text{ and } w\notin V(C_j).
\end{split}
\end{align}

Lastly, we need the local connectivities $\kappa(c_j^a,x_i^b)$ for $a,b\in\{1,2\}$.
First, consider the case where $S$ consists only of a single clause $C_j$.
Then we obtain
\begin{equation}\label{eq:x-c-singleclause}
    \kappa(c_j^a,x_i^b)=\begin{cases}
        3 &\text{if } (a,b)=(2,1) \text{ and } X_i \text{ is a positive literal in } C_j,\\
        3 &\text{if } (a,b)=(1,2) \text{ and } X_i \text{ is a negative literal in } C_j,\\
        2 &\text{otherwise}.
    \end{cases}
\end{equation}
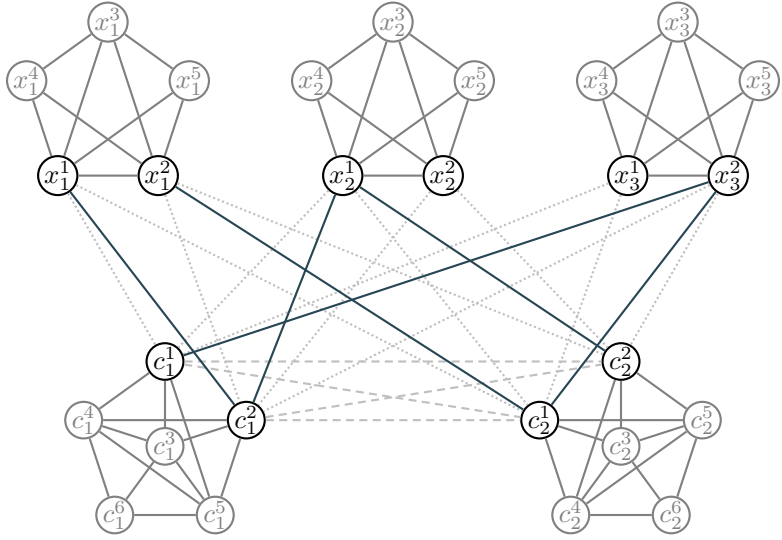
\begin{figure}
    \centering
    \begin{tikzpicture}[scale=0.75]

    \begin{scope}[shift={(1,0)}, rotate=18]
    \node[labeledvertex, gray] (x14) at (144:1.5) {\footnotesize{$x_1^4$}};
    \node[labeledvertex, gray] (x13) at (72:1.5) {\footnotesize{$x_1^3$}};
    \node[labeledvertex, gray] (x15) at (0:1.5) {\footnotesize{$x_1^5$}};
    \node[labeledvertex] (x12) at (288:1.5) {\footnotesize{$x_1^2$}};
    \node[labeledvertex] (x11) at (216:1.5) {\footnotesize{$x_1^1$}};

    \draw[edge, gray] (x11) -- (x12);
    \draw[edge, gray] (x11) -- (x13);
    \draw[edge, gray] (x11) -- (x14); 
    \draw[edge, gray] (x11) -- (x15);
    \draw[edge, gray] (x12) -- (x13);
    \draw[edge, gray] (x12) -- (x14);
    \draw[edge, gray] (x12) -- (x15);
    \draw[edge, gray] (x13) -- (x14);
    \draw[edge, gray] (x13) -- (x15);
    \end{scope}

    \begin{scope}[shift={(6,0)}, rotate=18]
    \node[labeledvertex, gray] (x24) at (144:1.5) {\footnotesize{$x_2^4$}};
    \node[labeledvertex, gray] (x23) at (72:1.5) {\footnotesize{$x_2^3$}};
    \node[labeledvertex, gray] (x25) at (0:1.5) {\footnotesize{$x_2^5$}};
    \node[labeledvertex] (x22) at (288:1.5) {\footnotesize{$x_2^2$}};
    \node[labeledvertex] (x21) at (216:1.5) {\footnotesize{$x_2^1$}};

    \draw[edge, gray] (x21) -- (x22);
    \draw[edge, gray] (x21) -- (x23);
    \draw[edge, gray] (x21) -- (x24); 
    \draw[edge, gray] (x21) -- (x25);
    \draw[edge, gray] (x22) -- (x23);
    \draw[edge, gray] (x22) -- (x24);
    \draw[edge, gray] (x22) -- (x25);
    \draw[edge, gray] (x23) -- (x24);
    \draw[edge, gray] (x23) -- (x25);
    \end{scope}

    \begin{scope}[shift={(11,0)}, rotate=18]
    \node[labeledvertex, gray] (x34) at (144:1.5) {\footnotesize{$x_3^4$}};
    \node[labeledvertex, gray] (x33) at (72:1.5) {\footnotesize{$x_3^3$}};
    \node[labeledvertex, gray] (x35) at (0:1.5) {\footnotesize{$x_3^5$}};
    \node[labeledvertex] (x32) at (288:1.5) {\footnotesize{$x_3^2$}};
    \node[labeledvertex] (x31) at (216:1.5) {\footnotesize{$x_3^1$}};

    \draw[edge, gray] (x31) -- (x32);
    \draw[edge, gray] (x31) -- (x33);
    \draw[edge, gray] (x31) -- (x34); 
    \draw[edge, gray] (x31) -- (x35);
    \draw[edge, gray] (x32) -- (x33);
    \draw[edge, gray] (x32) -- (x34);
    \draw[edge, gray] (x32) -- (x35);
    \draw[edge, gray] (x33) -- (x34);
    \draw[edge, gray] (x33) -- (x35);
    \end{scope}

    \begin{scope}[shift={(2,-6)}, rotate=162]

    \node[labeledvertex, gray] (c14) at (144:1.5) {\footnotesize{$c_1^5$}};
    \node[labeledvertex, gray] (c13) at (0,0) {\footnotesize{$c_1^3$}};
    \node[labeledvertex, gray] (c15) at (0:1.5) {\footnotesize{$c_1^4$}};
    \node[labeledvertex] (c11) at (288:1.5) {\footnotesize{$c_1^1$}};
    \node[labeledvertex] (c12) at (216:1.5) {\footnotesize{$c_1^2$}};
    \node[labeledvertex, gray] (c16) at (72:1.5) {\footnotesize{$c_1^6$}};
    
    \draw[edge, gray] (c11) -- (c13);
    \draw[edge, gray] (c11) -- (c14);
    \draw[edge, gray] (c11) -- (c15); 
    \draw[edge, gray] (c12) -- (c13);
    \draw[edge, gray] (c12) -- (c14);
    \draw[edge, gray] (c12) -- (c15);
    \draw[edge, gray] (c13) -- (c14);
    \draw[edge, gray] (c13) -- (c15);
    \draw[edge, gray] (c14) -- (c15);
    \draw[edge, gray] (c13) -- (c16);
    \draw[edge, gray] (c14) -- (c16);
    \draw[edge, gray] (c15) -- (c16);
    \end{scope}

    \begin{scope}[shift={(10,-6)}, rotate=234]

    \node[labeledvertex, gray] (c24) at (144:1.5) {\footnotesize{$c_2^5$}};
    \node[labeledvertex, gray] (c23) at (0,0) {\footnotesize{$c_2^3$}};
    \node[labeledvertex, gray] (c25) at (0:1.5) {\footnotesize{$c_2^4$}};
    \node[labeledvertex] (c21) at (288:1.5) {\footnotesize{$c_2^1$}};
    \node[labeledvertex] (c22) at (216:1.5) {\footnotesize{$c_2^2$}};
    \node[labeledvertex, gray] (c26) at (72:1.5) {\footnotesize{$c_2^6$}};
    
    \draw[edge, gray] (c21) -- (c23);
    \draw[edge, gray] (c21) -- (c24);
    \draw[edge, gray] (c21) -- (c25); 
    \draw[edge, gray] (c22) -- (c23);
    \draw[edge, gray] (c22) -- (c24);
    \draw[edge, gray] (c22) -- (c25);
    \draw[edge, gray] (c23) -- (c24);
    \draw[edge, gray] (c23) -- (c25);
    \draw[edge, gray] (c24) -- (c25);
    \draw[edge, gray] (c23) -- (c26);
    \draw[edge, gray] (c24) -- (c26);
    \draw[edge, gray] (c25) -- (c26);
    \end{scope}

    \draw[edge, lightgray, densely dotted] (x11) -- (c11);
    \draw[edge, lightgray, densely dotted] (x12) -- (c12);
    \draw[edge, lightgray, densely dotted] (x11) -- (c21);
    \draw[edge, lightgray, densely dotted] (x12) -- (c22);
    \draw[edge, lightgray, densely dotted] (x21) -- (c11);
    \draw[edge, lightgray, densely dotted] (x22) -- (c12);
    \draw[edge, lightgray, densely dotted] (x21) -- (c21);
    \draw[edge, lightgray, densely dotted] (x22) -- (c22);
    \draw[edge, lightgray, densely dotted] (x31) -- (c11);
    \draw[edge, lightgray, densely dotted] (x32) -- (c12);
    \draw[edge, lightgray, densely dotted] (x31) -- (c21);
    \draw[edge, lightgray, densely dotted] (x32) -- (c22);

    \draw[edge, lightgray, densely dashed] (c11) -- (c21);
    \draw[edge, lightgray, densely dashed] (c11) -- (c22);
    \draw[edge, lightgray, densely dashed] (c12) -- (c21);
    \draw[edge, lightgray, densely dashed] (c12) -- (c22);
    
    \draw[edge, navy] (x11) -- (c12);
    \draw[edge, navy] (x12) -- (c21);
    \draw[edge, navy] (x21) -- (c12);
    \draw[edge, navy] (x21) -- (c22);
    \draw[edge, navy] (x32) -- (c11);
    \draw[edge, navy] (x32) -- (c21);
    
    \end{tikzpicture}
    \caption{The graph $G(S)$ for $S= (X_1\vee X_2\vee \overline{X}_3)\wedge(\overline{X}_1\vee X_2\vee \overline{X}_3)$. The solid edges between gadgets are those that depend on whether a variable occurs as a positive or negative literal in a respective clause}
    \label{fig:ClauseAndVariables}
\end{figure}%
Now consider an arbitrary instance $S$. Every clause $C_k$ in which $X_i$ occurs yields additional $c_j^a$-$x_i^b$ paths.
Namely, the two independent paths $(c_j^b,c_k^1,x_i^1)$ and $(c_j^b,c_k^2,x_i^1)$ if $X_i$ is a positive literal in $C_k$.
If $X_i$ is a negative literal in $C_k$, then $x_i^1$ is not adjacent to $c_k^2$, but we still find the path $(c_j^1,c_k^1,x_i^1)$.
There are not more independent paths, because without using the vertices $c_k^a$, $k\neq j$, that are adjacent to $x_i^1$, we are left with the path count in \cref{eq:x-c-singleclause}. 
Analogous arguments hold for $x_i^2$.
Denoting by $\alpha_i$ the number of clauses in which $X_i$ occurs as a positive literal, and by $\beta_i$ the number of clauses where it occurs as a negative literal, we summarize
\begin{equation}\label{x-c-multipleclause}
    \kappa(c_j^a,x_i^b)=\begin{cases}
        3+2\alpha_i+\beta_i &\text{if } (a,b)=(2,1) \text{ and } X_i \text{ is a positive literal in } C_j,\\
        2+2\alpha_i+\beta_i &\text{if } (a,b)\neq(2,1) \text{ and } X_i \text{ is a positive literal in } C_j,\\
        3+\alpha_i+2\beta_i &\text{if } (a,b)=(1,2) \text{ and } X_i \text{ is a negative literal in } C_j,\\
        2+\alpha_i+2\beta_i &\text{if } (a,b)\neq(1,2) \text{ and } X_i \text{ is a negative literal in } C_j.
    \end{cases}
\end{equation}

We are now equipped to investigate the connectivity dimension of $G(S)=(V(S),E(S))$.

\begin{lemma}\label{clauseexactly2}
    Let $C_j$ be an arbitrary clause in $S$. 
    Then any connectivity basis of $G(S)$ contains exactly two of the vertices $c_j^3$, $c_j^4$ and $c_j^5$.
\end{lemma}    

\begin{proof}
    Because the elements of $\{c_j^3,c_j^4,c_j^5\}\eqqcolon A$ are pairwise twins, a resolving set of $G(S)$ has to contain \emph{at least} two of them, say $a$ and $b$.
    The third vertex $c$ of $A$ does not help distinguishing any other vertex.
    Using  Formulas (\ref{eq:sameclagadget}) and (\ref{eq:diffgadget}), we find that $\kappa(c,a)=5\neq\kappa(v,a)$ for all $v\notin A$.
    Therefore, $c$ itself is distinguished from any other vertex.
    Hence, a connectivity basis contains \emph{at most} two elements of $A$.
\end{proof}

\begin{lemma}\label{variableatleast1}
    Let $X_i$ be an arbitrary variable in $S$. 
    Then any connectivity basis of $G(S)$ contains at least one of the vertices $x_i^4$ and $x_i^5$.
\end{lemma}

\begin{proof}
    This is because $x_i^4$ and $x_i^5$ are twins.
\end{proof}

\begin{lemma}\label{variableatleast2}
    Let $X_i$ be an arbitrary variable in $S$. 
    Then any connectivity basis of $G(S)$ contains at least two vertices from $V(X_i)$. 
\end{lemma}

\begin{proof}
    Consider a connectivity basis $B$ of $G(S)$ and a variable $X_i$ from $S$. 
    By \cref{variableatleast1}, we know that $B$ has to contain $x_i^4$ or $x_i^5$, say $x_i^5$. 
    For a contradiction, assume that $B$ contains no other vertex of $V(X_i)$.
    Let $w\in B\setminus \{x_i^5\}$. Using Formula (\ref{eq:diffgadget}), we find that $\kappa(x_i^3,w)=\kappa(x_i^4,w)=2$, as $w\notin V(X_i)$. Also, by Formula (\ref{eq:samevargadget}), $\kappa(x_i^3,x_i^5)=\kappa(x_i^4,x_i^5)=3$. So indeed, $\kappa(x_i^3,w)=\kappa(x_i^4,w)$ for all $w\in B$. 
    In other words, $B$ is not resolving.
\end{proof}

\begin{corollary}\label{cdimatleast2mn}
    The connectivity dimension of $G(S)$ is at least $2(m+n)$.    
\end{corollary}

\begin{lemma}\label{cdimequal2mn}
    If $S$ is satisfiable, then $\cdim(G(S))=2(m+n)$.
\end{lemma}

\begin{proof}
    By \cref{cdimatleast2mn}, we know that $\cdim(G(S))\geq2(m+n)$.
    Fix a satisfying assignment for $S$. 
    For each clause $C_j$ let $c_j^4,c_j^5\in B$. 
    For each variable $X_i$ let $x_i^5\in B$.
    For each variable $X_i$ with assigned value \texttt{true} let $x_i^1\in B$.
    For each variable $X_i$ with assigned value \texttt{false} let $x_i^2\in B$.
    We show that $B$ is resolving for $G(S)$ and therefore $\cdim(G(S))\leq2(m+n)$. For this, we find for all vertices $v_1,v_2\in V(S)\setminus B$, $v_1\neq v_2$, a vertex $w\in B$ that distinguishes $v_1$ and $v_2$.
    
    Let $v_1\in V(C_j)$ and $v_2\in V(C_k)$ for some clauses $C_j$ and $C_k$, $k\neq j$.
    Using $c_j^5\in B$ and Formulas (\ref{eq:sameclagadget}) and (\ref{eq:diffgadget}), we find that $\kappa(u,c_j^5)\geq3 \neq 2=\kappa(v,c_j^5)$.

    Let $v_1\in V(X_i)$ and $v_2\in V(X_k)$ for some variables $X_i$ and $X_k$, $i\neq k$.
    Using $x_i^5\in B$ and Formulas (\ref{eq:samevargadget}) and (\ref{eq:diffgadget}), we find that $\kappa(u,x_i^5)=3\neq2=\kappa(v,x_i^5)$.

    Let $v_1\in V(X_i)$ for some variable $X_i$ and $v_2\in V(C_j)$ for some clause $C_j$. 
    Using $x_i^5\in B$ and Fomulas (\ref{eq:samevargadget}) and (\ref{eq:diffgadget}), we find that $\kappa(v_1,x_i^5)=3\neq2=\kappa(v_2,x_i^5)$.
    
    Remaining are the cases where $v_1$ and $v_2$ are in the same gadget. 
    Let $v_1,v_2\in V(X_i)$ for some variable $X_i$. 
    We have $x_i^5\in B$ and $x_i^1\in B$ or $x_i^2\in B$. 
    According to Formula (\ref{eq:samevargadget}), we have $\kappa(x_i^1,x_i^2)>4$, $\kappa(x_i^3,x_i^1)=\kappa(x_i^3,x_i^2)=4$ and $\kappa(x_i^4,x_i^1)=\kappa(x_i^4,x_i^2)=3$. 
    So either way, $v_1$ and $v_2$ are distinguished.
    
    Let $v_1,v_2\in V(C_j)$ for some clause $C_j$. 
    According to Formula (\ref{eq:sameclagadget}), we have $c_j^4\in B$ and $\kappa(c_j^1,c_j^4)=\kappa(c_j^2,c_j^4)=4$, $\kappa(c_j^3,c_j^4)=5$, $\kappa(c_j^6,c_j^4)=3$.
    The only pair of vertices $\{v_1,v_2\}\subset V(C_j)\setminus B$ that is not distinguished by $c_j^4$ is $\{c_j^1,c_j^2\}$. 
    Because we are given a satisfying assignment for $S$, in $C_j$ there is at least one positive literal set to \texttt{true} or a negative literal set to \texttt{false}.
    
    Let $X_i$ occur as positive literal in $C_j$ with assigned value \texttt{true}. 
    So, by construction, $x_i^1\in B$.
    By \cref{x-c-multipleclause}, we have $\kappa(c_j^1,x_i^1)=2+2\alpha_i+\beta_i\neq3+2\alpha_i+\beta_i=\kappa(c_j^2,x_i^1)$.

    Let $X_i$ occur as negative literal in $C_j$ with assigned value \texttt{false}. 
    So, by construction $x_i^2\in B$. 
    By \cref{x-c-multipleclause}, we have $\kappa(c_j^1,x_i^2)=3+\alpha_i+2\beta_i\neq2+\alpha_i+2\beta_i=\kappa(c_j^2,x_i^2)$.
\end{proof}

\begin{lemma}\label{cdimifsatisfiable}
    If $\cdim(G(S))=2(m+n)$, then $S$ is satisfiable.
\end{lemma}

\begin{proof}
    Let $B$ of size $|B|=2(m+n)$ be a resolving set for $G(S)$. 
    Due to \cref{clauseexactly2,variableatleast2}, we know that $B$ contains exactly two elements from each gadget. 
    Due to \cref{clauseexactly2,variableatleast1}, we can assume that for each clause $C_j$ we have $\{c_j^4,c_j^5\}\subseteq B$ and that for each variable $X_i$ we have $x_i^5\in B$.
    Note that by Formulas (\ref{eq:sameclagadget}) and (\ref{eq:diffgadget}), the landmarks $c_k^4,c_k^5$ and $x_k^5$ can not distinguish $c_j^1$ from $c_j^2$.
    
    There is one more landmark in $V(X_i)$ for all variables $X_i$.
    Using Formulas (\ref{eq:samevargadget}) and (\ref{eq:diffgadget}), one checks that, given $x_i^5\in B$, whenever there is $x_i^3\in B$ or $x_i^4\in B$, then $B'=B\setminus\{x_i^3,x_i^4\}\cup\{x_i^a\}$ is a resolving set of size $|B'|=2(m+n)$ for both $a=1$ and $a=2$.
    We may therefore assume that either $x_i^1\in B$ or $x_i^2\in B$.
    
    Let $X_i$ occur as positive literal in $C_j$. 
    By \cref{x-c-multipleclause}, we have that $x_i^1$ distinguishes $c_j^1$ from $c_j^2$ as $\kappa(c_j^1,x_i^1)=2+2\alpha_i+\beta_i\neq3+2\alpha_i+\beta_i=\kappa(c_j^2,x_i^1)$. 
    We also have $\kappa(c_j^1,x_i^2)=\kappa(c_j^2,x_i^2)=2+2\alpha_i+\beta_i$, so $x_i^2$ does not distinguish $c_j^1$ from $c_j^2$.

    Let $X_i$ occur as negative literal in $C_j$. 
    By \cref{x-c-multipleclause}, we have that $x_i^2$ distinguishes $c_j^1$ from $c_j^2$ as $\kappa(c_j^1,x_i^2)=3+\alpha_i+2\beta_i\neq2+\alpha_i+2\beta_i=\kappa(c_j^2,x_i^2)=2$. 
    We also have $\kappa(c_j^1,x_i^1)=\kappa(c_j^2,x_i^1)=2+\alpha_i+2\beta_i$, so $x_i^1$ does not distinguish $c_j^1$ from $c_j^2$.

 It follows that for each clause $C_j$ there has to be a variable $X_i$ occuring as positive literal in $C_j$ with $x_i^1\in B$ or occuring as negative literal in $C_j$ with $x_i^2\in B$. 
    In other words, given a resolving set $B$ of size $2(m+n)$, the assignment that assigns to $X_i$ the value \texttt{true} if $x_i^1\in B$ and \texttt{false} otherwise, satisfies $S$.
\end{proof}

\begin{theorem}
    Deciding whether $\cdim(G)\leq k$, where $G$ is a graph and $k$ an integer, is NP-complete.
\end{theorem}

\begin{proof}
    The problem is in NP as verifying whether a given set is resolving reduces to computing the respective local connectivities. 
    Those can be determined by standard maximum flow algorithms efficiently.
    Given a clause $S$ for 3-SAT on $m$ clauses and $n$ variables, deciding whether $\cdim(G(S))\leq 2(m+n)$ solves also 3-SAT by \cref{cdimequal2mn} and \cref{cdimifsatisfiable}.
\end{proof}

\section{Conclusions and outlook}\label{sec:outlook}

This article is an invitation to study a new localization concept, a graph's connectivity dimension.
Beginning with the elementary bound $0\leq \cdim(G)\leq n-1$ on the connectivity dimension of a graph $G$, \cref{sec:prescribed} provides characterizations for the cases where $\cdim(G)\in\{0,1,n-1\}$.
Beyond the result of \cref{thm:cdimthreshold}, stating that for given $k\geq 2$, $k\in\mathbb{N}$, it is possible to construct infinite families of graphs achieving connectivity dimension $k$, a natural further goal could be to find compact characterizations for graphs of certain dimension. Two cases stand out as particularly interesting.

\begin{problem}\label{op:cdim2andn-2}
Give a complete characterization of graphs $G$ with $\cdim(G)=2$ or $\cdim(G)=n-2$.    
\end{problem}

To address this problem, but also for its own sake, it might be instructive to study the connectivity dimension under constraints on a graph's regularity, connectivity, or diameter. 
A solution to this problem could also be valuable to improve the bound from \cref{thm:cdimblocksratio}. In its proof, a reduction of a minimal counterexample is carried out by removing non-trivial blocks having $\cdim\geq2$. 
Knowing more about the structure of blocks of $\cdim=2$ could strengthen the reduction arguments involved.

\cref{sec:blockstructure} addresses relationships between a graph's connectivity dimension and its block structure.
A key result is the bound $\smash{\frac{\cdim(G)}{\blocks(G)}\geq \frac{1}{2}}$, established in \cref{thm:cdimblocksratio}.
Yet, the best concrete construction known to us only achieves a ratio of $\frac{2}{3}$.
Closing this gap remains a natural open problem.

\begin{problem}\label{op:cdimblocksratio}
Find a lower bound on $\smash{\frac{\cdim(G)}{\blocks(G)}}$ larger than than $\frac{1}{2}$ or provide a family of graphs where $\smash{\frac{\cdim(G)}{\blocks(G)}}$ is lower than $\frac{2}{3}$.
\end{problem}

In the proof of \cref{thm:cdimblocksratio}, special care needs to be taken for graphs forcing a $\mathds{1}$ representation.
Particular instances of such graphs are graphs that contain uniformly \mbox{$1$-connected} vertices, where we call a vertex $v$ uniformly $k$-connected if $\kappa(v,w)=k$ for all other vertices $w$.
For example, all leaves of a graph are uniformly \mbox{$1$-connected} and in a cycle graph, all vertices are uniformly \mbox{$2$-connected}.
Besides inherent relevance in the problems discussed here, it might be of interest to study and characterize uniformly connected vertices, or vertices with special connectivity patterns, in general. In particular, if one would be interested in generalizing \cref{lem:disjointunion,lem:cdimHjoin} to a situation where $H$ is $k$-connected, one would have to take special care of uniformly $\ell$-connected vertices for $\ell\leq k$.

\cref{sec:complexity} shows that it is NP-complete to determine a graph's connectivity dimension.
As discussed in the introduction, this problem can also be phrased as a special set cover problem, which admits a logarithmic-factor approximation via the greedy algorithm.
One might ask whether it is possible to design better heuristics tailored to the connectivity dimension.

Another research direction could be the study of dimension concepts that build on different notions of connectivity.
Instead of the local connectivity, the resistance distance or local edge-connectivity likewise give rise to meaningful invariants.

\section*{Acknowledgments}
This work was supported by the Croatian Science Foundation under the project number HRZZ-IP-2024-05-2130.
The second author gratefully acknowledges partial support by the Deutsche Forschungsgemeinschaft (DFG, German Research Foundation) under Germany's Excellence Strategy -- The Berlin Mathematics Research Center MATH+ (EXC-2046/1, project ID: 390685689).

\bibliographystyle{plain}
\bibliography{bibliography}

\end{document}